\documentclass[a4paper,11pt]{article}
\usepackage[latin1]{inputenc}
\usepackage[english]{babel}

\usepackage{latexsym}
\usepackage{amsfonts}
\usepackage[latin1]{inputenc}
\usepackage{graphicx}
\usepackage{amsmath}
\usepackage{amsthm}

\usepackage{amssymb}
\usepackage[all]{xy}
\usepackage[normalem]{ulem}
\usepackage{lscape}
\usepackage{color}

\usepackage{enumitem}

\usepackage{times} 

\usepackage{mathabx}

\usepackage{adjustbox}

\usepackage{MnSymbol}

\newcommand{\rojo}[1]{\textcolor[rgb]{1.00,0.00,0.00}{#1}}

\newcommand{\prt}[1]{\langle #1 \rangle}

\newcommand{\ZZ}{\mathbb{Z}}
\newcommand{\NN}{\mathbb{N}}

\newtheorem{defi}{Definition}
\newtheorem{teo}{Theorem}
\newtheorem{exem}{Example}
\newtheorem{obs}{Remark}
\newtheorem{prop}{Proposition}
\newtheorem{lem}{Lemma}
\newtheorem{coro}{Corollary}

\begin{document}


\title{Aggregation functions on  $n$-dimensional ordered vectors  equipped with an  admissible order  and an application in  multi-criteria group decision-making}

\author{Thadeu Milfont and Ivan Mezzomo\\ \{thadeuribeiro,imezzomo\}@ufersa.edu.br \\ Departamento de Ci\^encias Naturais, Matem\'atica e Estat\'istica -- DCME \\ 
Universidade Federal Rural de Semi-\'Arido -- UFERSA\\
Mossor\'o, Rio Grande do Norte, Brazil\\[0.4cm] Benjam\'in Bedregal\\bedregal@dimap.ufrn.br\\Departamento de Inform\'atica e Matem\'atica Aplicada - DIMAp \\ Universidade Federal do Rio Grande do Norte - UFRN\\[0.4cm] Edmundo Mansilla\\edmundo.mansilla@umag.cl\\ Departamento de Matem\'aticas y F\'isica-- DMF \\ Universidad de Magallanes -- UMAG\\ Punta Arenas,  Chile \\[0.4cm]  Humberto Bustince \\ bustince@upna.es \\ Departamento de Estad\'istica,  Inform\'atica e Matem\'aticas \\
    Universidad Publica de Navarra \\
    Pamplona, Navarra,  Spain}

%
%
%
%
%
%

\maketitle

\begin{abstract}
$n$-Dimensional fuzzy sets are a fuzzy set extension where the membership values are $n$-tuples of real numbers in the unit interval $[0,1]$ increasingly ordered, called $n$-dimensional intervals. The set of $n$-dimensional intervals is denoted by $L_n([0,1])$. This paper aims to investigate semi-vector spaces over a weak semifield and aggregation functions concerning an admissible order on the set of $n$-dimensional intervals and the construction of aggregation functions on $L_n([0,1])$ based on the operations of the semi-vector spaces. In particular, extensions of the family of OWA and weighted average aggregation functions are investigated. Finally, we develop a multi-criteria group decision-making method based on $n$-dimensional aggregation functions with respect to an admissible order and give an illustrative example.
\end{abstract}

%

\paragraph{[Keywords:} $n$-dimensional fuzzy sets, ordered semi-vector spaces, admissible orders, aggregation function, ordered weighted averaging, decision making.

\section{Introduction}

A mathematical function is a rule, which takes one or several inputs and returns an output such that for each input has a unique output associated to it,  average aggregation functions are functions $A: [0,1]^n \rightarrow [0,1]$ with special properties that takes real inputs from a closed interval $[0,1]$ and aggregates them in a single real value   which, in some sense, represents all the other values. One example of aggregation function is the   ordered weighted averaging (OWA) introduced by Yager in \cite{Yag88}.
Aggregation functions are widely used in physics and natural sciences such as pure and applied mathematics, computer science, engineering, economics, among others. An extensive study on average aggregation functions can be found in \cite{Beliakov-Bustince,Calvo}.

In recent years the interest of researchers in generalizations and extensions of fuzzy set theory \cite{Zadeh} has grown, considering both the different ways of interpreting the uncertainty of phenomena and the scope of the problem. In \cite{Bustince3} it is given a historical and hierarchical analysis of the most important of such extensions, among which we can mention the interval-valued fuzzy sets \cite{Peka19,Sambuc75}, hesitant fuzzy sets \cite{Rosa,Torra}, $L$-fuzzy sets \cite{Goguen} and $n$-dimensional fuzzy sets \cite{Benja2,Shang,Zanotelli20}.

In particular, this last extension considers as membership degrees  $n$-dimensional ordered vectors of real numbers in $[0,1]$, i.e.  elements of  $L_n([0,1]) =\{ (x_1, \ldots, x_n) \in [0,1]^n \ | \ x_1 \leq \cdots \leq x_n \}$. This kind of extension arises, for example, in situations where $n$ experts (human or not) provide their membership degrees of each element of the discourse universe to the same fuzzy set, and their identity is irrelevant (and therefore all they have the same weight) or yet necessary.
 Later, Bedregal et al. in \cite{Benja2} have investigated the $m$-ary $n$-dimensional aggregation functions with respect to the natural partial order (also known as product order) on the set $L_n([0,1])$ and in particular  the $n$-dimensional  t-norms. 
 
 Admissible orders are orders which refine the ``natural'' order on  a set of membership degrees in a fuzzy extensions. Originally they were introduced in the context of interval-valued fuzzy sets by H. Bustince et al. in \cite{Bustince1} and since then they have been widely used  \cite{Bustince20,Costa20,Santana20,Zapata}. Lately, such notion was studied in other types of fuzzy sets, such as interval-valued intuitionistic fuzzy sets \cite{Silva,Laura1,Laura2}, hesitant fuzzy sets \cite{Monica20,Monica21}, multidimensional fuzzy sets \cite{Annaxsuel20} and $n$-dimensional fuzzy sets \cite{Laura3}. 


The main contribution of this paper is to introduce and study the notion of n-dimensional aggregation functions with respect to an admissible order on  $L_n([0,1])$ of the $n$-dimensional ordered vectors on $[0,1]$. In particular, in order to extend the OWA and the weighted average aggregation function to this context, we must consider addition and product by an scalar (the weights) in $[0,1]$. Nevertheless, in order to preserve the most basic properties of the OWA and many other OWA-like operators, it is necessary  some kind of compatibility between  the admissible order  and  both operations. With this in mind, it was introduced the notion of ordered semi-vector spaces over a weak semifield, which is  generalization of ordered vector spaces \cite{Cri76}. Finally, $n$-dimensional  aggregation functions with respect to an admissible order were used to  generate a score in $L_n([0,1])$, where $n$ corresponds to the number of experts, for each alternative in a multiple criteria group decision making  problem and thereby a ranking of the  alternatives based on these scores and the admissible order.


%

To do so, the paper is organized as follows: some important concepts related to  aggregation functions and $n$-dimensional fuzzy sets are recalled in Section 2. Next section is devoted to introducing the notion of ordered semi-vector spaces over weak semifields in general and for the particular case of the $L_n([0,1])$ equipped with an admissible order. In Section 4  $n$-dimensional aggregation functions over  admissible orders are considered and it is proposed an OWA like operator for the admissibly ordered semi-vector space with $L_n([0,1])$ as the underlying set studied in previous section. Section 5 brings a multiple criteria and multiple expert decision making method based on $n$-dimensional aggregation functions over  admissible orders joint with an illustrative example.  Finally, Section 6 presents the final remarks and discusses further works.

\section{Preliminary Concepts}

In this section we recall some important concepts for our proposal in this work. For a detailed overview on these topics, we recommend \cite{Benja2,Beliakov-Bustince,Calvo,Shang}.

In the whole text we are considering $\mathbb{N}^*$ as the set of natural elements without the zero and $N^n_m=\{m,m+1,m+2, \dots,n\}$ for each $m,n\in \mathbb{N}^*$ such that $m\leq n$.
In addition, $\omega=(w_1,\ldots,w_n) \in [0,1]^{n}$ such that $\sum\limits_{i=1}^{n} w_i=1$ is called a  weighting vector. $\omega$ is strictly positive if  $w_i>0$ for each $i\in N_1^n$.

\subsection{$n$-Dimensional fuzzy sets}

In this subsection, we make a review of $n$-dimensional fuzzy sets over a nonempty referential set $X$. These sets arise as a natural generalization of the interval-valued fuzzy sets \cite{Peka19,Sambuc75}, three-dimensional fuzzy sets \cite{Li09} and interval-valued Atanassov intuitionistic fuzzy sets \cite{AG89} which mathematically are equivalent to four-dimensional fuzzy sets. This generalization is useful in situations where $n$ experts provide an evaluation of how much an alternative or element satisfies some criteria or condition and the identity of the expert (i.e. the order in which the evaluations are provided) is not relevant.

\begin{defi} \cite{Shang}
	Let $X$ be a non-empty set and $n \in \mathbb{N}^{*}$. According to \cite{Shang}, a set
	\begin{equation*}
	A = \{(x, \mu^{1}_{A}(x),\dots,\mu^{n}_{A}(x))\mid x \in X\}
	\end{equation*}
	is called an $n$-dimensional fuzzy set on $X$, given that for all $i \in N_1^n$ the membership functions $\mu^{i}_{A} : X \rightarrow [0,1]$ satisfy $\mu^{1}_{A}(x) \leq \cdots \leq \mu^{n}_{A}(x)$.
\end{defi}

 Bedregal et al. in \cite{Benja1} have defined the set
\begin{equation} \label{eqn: n-dimensional}
L_{n}([0,1]) = \{ (x_{1}, \dots, x_{n}) \in [0,1]^{n} \mid x_{1} \leq x_{2} \leq \dots \leq x_{n}\}.
\end{equation}
Since $L_{1}([0,1]) = [0,1]$ and $L_{2}([0,1]) = L([0,1])$ (the closed subintervals of the unit
	interval $[0,1]$), the elements of  $L_{n}([0,1])$ are called $n$-dimensional intervals (see \cite{Benja2}).

\begin{obs} \label{degem}
	The $i$-th projection on $L_{n}([0,1])$, for $i\in N_1^n$, is the function $\pi_{i} : L_{n}([0,1]) \rightarrow [0,1]$  given by $\pi_{i}(x_{1}, \dots, x_{n}) = x_{i}$. For each $x \in [0, 1]$ the element $(x, \ldots, x) \in L_{n}([0,1])$ is denoted by $/x/$ and it is called a degenerate element. The set of all degenerate elements in $L_{n}([0,1])$ is denoted by $\mathcal{D}_{n}$.
\end{obs}

Shang, Yuan and Lee, in \cite{Shang}, considered the following partial order on $L_{n}([0,1])$:
\begin{equation} \label{eqn: n-ordem}
\textbf{x} \leq_{n}^{p} \textbf{y} \Leftrightarrow \pi_{i}(\textbf{x}) \leq \pi_{i}(\textbf{y}) \text{ for each } i \in N_1^n,
\end{equation}
called the product order.

The poset $\langle L_{n}([0,1]) , \leq_{n}^{p} \rangle$ is a continuous lattice \cite{Benja2} and hence a complete distributive lattice, having $/0/$ and $/1/$ as its smallest and greatest elements, respectively (see \cite{Gierz, Scott}). In particular, the infimum and supremum of a set $S\subseteq L_n([0,1])$ in this lattice,  denoted by $\bigwedge S$ and $\bigvee S$, are given by

$$\bigwedge S=(\min\limits_{\mathbf{x}\in S} \pi_1(\mathbf{x}),\ldots,\min\limits_{\mathbf{x}\in S} \pi_n(\mathbf{x}))$$
and
$$\bigvee S=(\max\limits_{\mathbf{x}\in S} \pi_1(\mathbf{x}),\ldots,\max\limits_{\mathbf{x}\in S} \pi_n(\mathbf{x}))$$

In \cite{Laura3}, De Miguel et al. have introduced the concept of an admissible order on $L_n([0,1])$ with respect to the product order $\leq_{n}^{p}$.

\begin{defi}
 A total order $\preceq$ on $L_{n}([0,1])$ is called admissible if for each $\mathbf{x},\mathbf{y}\in L_n([0,1])$ we have that $\mathbf{x}\preceq \mathbf{y}$ whenever $\mathbf{x}\leq_n^p \mathbf{y}$ .
 \end{defi}

\begin{exem} \label{ex-adm-orders} Let $\mathbf{x},\mathbf{y}\in L_n([0,1])$ and $\tau:N_1^n\rightarrow N_1^n$ be a bijection.
	\begin{enumerate}
	 \item \cite[Example 2]{Laura3}
	 	The reflexive closure of
	 $$\mathbf{x} \prec_{\tau} \mathbf{y}\Leftrightarrow \exists k\in N_1^n\mbox{ s.t. }\pi_{\tau(k)}(\mathbf{x})<\pi_{\tau(k)}(\mathbf{y})\mbox{ and }\forall i, 1\leq i < k, \pi_{\tau(i)}(\mathbf{x})=\pi_{\tau(i)}(\mathbf{y})$$
	 is an admissible order.
	 In particular, if  $\tau$ is the identity then $\preceq_\tau$ corresponds to the lexicographical order and if $\tau(i)=n-i+1$ then  $\preceq_\tau$ is the anti-lexicographical order.

	 \item Let  $\omega=(w_1,\ldots,w_n)$ be a weighting vector and $$F_\omega(\mathbf{x},\mathbf{y})=\sum\limits_{i=1}^n w_i \max(0,\pi_i(\mathbf{x})-\pi_i(\mathbf{y})).$$ Then $$\mathbf{x} \preceq^\omega_\tau \mathbf{y}\Leftrightarrow F_\omega(\mathbf{x},\mathbf{y}) < F_\omega(\mathbf{y},\mathbf{x})\mbox{ or }(F_\omega(\mathbf{x},\mathbf{y})=F_\omega(\mathbf{y},\mathbf{x})\mbox{ and }\mathbf{x} \preceq_{\tau} \mathbf{y}),$$
	 is an admissible order.

	 \item   Let $A$ be an $n$-ary  aggregation function. Then
	 $$\mathbf{x} \preceq_A^\tau \mathbf{y}\Leftrightarrow A(\mathbf{x}) < A(\mathbf{y})\mbox{ or }(A(\mathbf{x})=A(\mathbf{y})\mbox{ and }\mathbf{x} \preceq_{\tau} \mathbf{y})$$  is an admissible order.
	\end{enumerate}

\end{exem}

Since each admissible order $\preceq$ is total then any finite set $S\subseteq L_n([0,1])$ has a minimum and a maximum with respect to the order. We will denote such elements as $\bigcurlywedge S$ and $\bigcurlyvee S$, respectively, and as $\bigcurlywedge\limits_{j=1}^m \mathbf{x}_j$ and $\bigcurlyvee\limits_{j=1}^m \mathbf{x}_j$ when $S=\{\mathbf{x}_1,\ldots,\mathbf{x}_m\}$.

\section{Ordered Semi-Vector Spaces}

Semi-vector spaces were introduced in \cite{PS74} by Prem Prakash and Murat Sertel as  a set of elements, called vectors, that is closed under finite vector addition and scalar multiplication satisfying some conditions.  In this definition, it is considered as scalar set just the set of  non-negative real numbers with their usual addition and product and the property of existence of a neutral element for the vector addition is not considered. In   \cite{Kan93}, Vasantha Kandasamy introduced a more general and complete definition for semi-vector spaces  (SVS) and in this section we introduce definitions of SVS and ordered SVS which are related to a weak semifield instead of a semifield.

\subsection{Ordered semi-vector spaces over a weak semifield}

A semifield is an algebraic structure together with two binary operations, addition and multiplication. It is similar to a field but with some axioms removed  or relaxed. There are different non equivalent definitions of semifield (see for example \cite{HW96,Kan93,LMP15}). In particular, the given in \cite{Kan93} does not encompass the notion of field. In this paper we will adopt  a weak notion of semifield  in order to guarantee that each field is also a semifield.

\begin{defi} A weak semifield is an algebra $\mathcal{F}=\prt{F,\pluscirc,\cdot}$  satisfying the following axioms for each $r,s,t\in F$:
 \begin{enumerate}
  \item $r\pluscirc (s\pluscirc t)=(r\pluscirc s)\pluscirc t$ and $r\cdot (s\cdot t)=(r\cdot s)\cdot t$;
  \item $r\pluscirc s=s\pluscirc r$ and $r\cdot s=s\cdot r$;
  \item $r\cdot (s\pluscirc t)=(r\cdot s)\pluscirc(r\cdot t)$;
  \item $\exists 0\in F$ such that $r\pluscirc 0=r$;
   \item $\exists 1\in F$ such that $1\cdot r=r$.
 \end{enumerate}
 $F^+$ we will denote $F-\{0\}$.
\end{defi}
 Observe that weak semifields are fields where the inverses of the addition and multiplication are dropped, or in other words, commutative semirings (in the sense of \cite{Kan02}) with a multiplicative identity element.

\begin{defi} A semi-vector space  $\mathcal{V}$ over a weak semifield $\mathcal{F}=\prt{F,\pluscirc,\cdot}$ is a nonempty set $V$, equipped with two operations, $+:V\times V \rightarrow V$ and $\star:F\times V\rightarrow V$, fulfilling the following axioms, for each $x,y,z\in V$ and $r,s\in F$:

\begin{enumerate}[labelindent=\parindent,  leftmargin=*,label=\normalfont{(SV\arabic*)}]
  \item $x+(y+z)=(x+y)+z$;
  \item $x+y=y+x$;
\item $r\star(s\star x)= (r\cdot s)\star x$;
  \item $1\star x=x$;
  \item $r\star (x+y)=(r\star x)+(r\star y)$;
  \item $(r\pluscirc s)\star x= (r\star x) + (s\star x)$;
   \item $\exists \mathbf{0}\in V$ such that $\mathbf{0}+x=x$ for each $x\in V$.
 \end{enumerate}
\end{defi}

In a semi-vector space, $\mathcal{V}=\prt{V,+,\star}$ the elements of $V$ are called vectors, the operation $\star$ is called scalar multiplication and $+$ is the vector addition.

In a semi-vector space the  vector $\mathbf{0}$ in axiom (SV7) is called zero vector or additive identity  vector. The zero vector is unique.



In mathematics, an ordered vector space  is a vector space equipped with a partial order that is compatible with the vector space operations \cite{Cri76}. Clearly, this notion can be extended for semi-vector spaces. In \cite{GG99}, the notion of an ordered semi-vector space was introduced, but only considering the semifield of non-negative real numbers. Here we extend it for an arbitrary weak semifield.

\begin{defi} Let $\mathcal{V}=\prt{V,+,\star}$   be a semi-vector space over a weak semifield $\mathcal{F}=\prt{F,\pluscirc}$ and $\leq$ be a partial order on $V$. We say that $\leq$ is compatible with the scalar product of $\mathcal{V}$ if
\begin{enumerate}[start=8,labelindent=\parindent, leftmargin=*,label=\normalfont{(SV\arabic*)}]
  \item  $r\star x\leq r\star y$ for each $x,y\in V$ such that $x\leq y$ and $r\in F$,
 \end{enumerate}
and is compatible with the vector addition of $\mathcal{V}$ if
\begin{enumerate}[resume*,start=9,labelindent=\parindent, leftmargin=*,label=\normalfont{(SV\arabic*)}]
  \item $x+z\leq y+z$ for each $x,y,z\in V$ such that $x\leq y$.
 \end{enumerate}
In addition, the pair $\prt{\mathcal{V},\leq}$ is an ordered semi-vector space over $\mathcal{F}$ or, equivalently, $\leq$ is compatible with $\mathcal{V}$, if it is compatible with both, the scalar product and the vector addition of $\mathcal{V}$.
\end{defi}

Notice that for any semi-vector space $\mathcal{V}$ over an arbitrary weak semifield, the pair  $\prt{\mathcal{V},=}$  is an ordered semi-vector space, which is called trivial.

 The following proposition guarantees that each  semi-vector space determines a natural partial order which is compatible with the semi-vector space.

 \begin{prop}\label{pro-nat-ord}
   Let $\mathcal{V}=\prt{V,+,\star}$ be a semi-vector space and consider the binary relation defined for each $x,y\in  V$ as
\begin{equation}\label{eq-Nat-orde-SVS}
 x\leq_{V} y \Leftrightarrow \exists z\in V,\, x+z=y.
\end{equation}
If $\leq_V$ is antisymmetric, then  $\prt{\mathcal{V},\leq_{V}}$ is an ordered semi-vector space.
 \end{prop}
 \begin{proof} First of all, observe that by (SV7) and (SV2), $\leq_V$ is reflexive. If $x\leq_V y$ and $y\leq_V z$, then by Eq. (\ref{eq-Nat-orde-SVS}), $x+u=y$ and $y+v=z$ for some $u,v\in V$. So, $x+(u+v)=z$ and therefore, $x\leq_V z$, i.e. $\leq_V$ is transitive.  Hence, $\leq_V$ is a partial order on $V$ whenever $\leq_V$ is antisymmetric.
  Let $\mathcal{V}=\prt{V,+,\star}$ be a semi-vector space , $x,y\in  V$. Then, we just need to prove (SV8) and (SV9). So,

  \begin{eqnarray*}
  x\leq_{V} y & \Leftrightarrow & \exists z\in V,\, x+z=y \ \ \ \mbox{by Eq. (\ref{eq-Nat-orde-SVS})} \\
                    & \Leftrightarrow & \exists z\in V,\,\forall r\in F,\, r\star (x + z) = r\star y \\
                    &\Leftrightarrow &  \exists z\in V,\,\forall r\in F,\, (r\star x) + (r\star z) = r\star y \ \ \ \mbox{by (SV5)}\\
                    & \Rightarrow &  \forall r\in F,\, r\star x \leq_V r\star y \ \ \ \mbox{by Eq. (\ref{eq-Nat-orde-SVS})}.
  \end{eqnarray*}
  Now, let $x, y, w\in V$. Then

  \begin{eqnarray*}
  x\leq_{V} y & \Leftrightarrow & \exists z\in V,\, x+z=y \ \ \ \mbox{by Eq. (\ref{eq-Nat-orde-SVS})}\\
                    & \Leftrightarrow & \exists z\in V,\,\forall w\in V,\, (x+z) + w = y + w \\
                    &\Leftrightarrow &   \exists z\in V,\,\forall w\in V,\,  (x+w) + z = y + w \ \ \ \mbox{by (SV1) and (SV2)}\\
                    & \Leftrightarrow &  \forall w\in V,\, x+w \leq_{V} y+ w \ \ \ \mbox{by Eq. (\ref{eq-Nat-orde-SVS})}.
  \end{eqnarray*}
  Therefore, $\prt{\mathcal{V},\leq_{V}}$ is an ordered semi-vector space.
 \end{proof}

Observe that $\leq_V$ is not always  antisymmetric and therefore, a partial order. For example,  in the semi-vector space $\mathcal{Z}=\prt{\ZZ,+,\cdot}$ over the weak semifield $\prt{\NN,+,\cdot}$, where $\ZZ$ is the set of integers, $\NN$ the set of natural numbers and $+$ and $\cdot$ are the usual addition and product operations, $x\leq_{\ZZ} y$ for each $x,y\in\ZZ$ and therefore $\leq_\ZZ$ is not antisymmetric. In fact,  $\leq_V$ is a partial order if and only if, there is no $u,v\in V-\{\mathbf{0}\}$ such that $u+v=\mathbf{0}$.

\subsection{Aggregation Functions}

In this subsection, we will define and give some examples of $n$-ary aggregation functions, which are functions that map  values in $[0,1]$ into a single value in $[0,1]$ which in some sense represents all of them.

\begin{defi}\label{def: agregacao}  \cite{Beliakov-Bustince}
	Let $n \in \mathbb{N}$ with $ n \geq 2$. A function $A: [0,1]^{n} \rightarrow [0,1]$ is an $n$-ary aggregation function if:
	\begin{description}
		\item[i)] $A(0,\dots,0) = 0$ and $A(1,  \dots,1) = 1$;
		\item[ii)] If $x_{i} \leq y_{i}$ for each $i\in N_1^n$, then $A(x_{1}, \dots, x_{n}) \leq A(y_{1}, \dots, y_{n})$.
	\end{description}
	In addition,
	\begin{description}
		\item[iii)]	$A$ is strict if for each $i\in N_1^n$ and $x_1,\ldots,x_n,y\in [0,1]$, $A(x_1,\ldots,x_n) < A(x_1,\ldots, x_{i-1},y,x_{i+1},\ldots,x_n)$ whenever $x_i < y$;
		\item[iv)] If there is a  positive real number $k$ such that for each $x_1,\ldots,x_n,\lambda\in [0,1]$, $A(\lambda x_1,\ldots,\lambda x_n)=\lambda^k A(x_1,\ldots,x_n)$ then $A$ is said to be homogeneous of order $k$; and
		\item[v)] $A$ is called internal if $A(x_1,\ldots,x_n)\in \{x_1,\ldots,x_n\}$ for each $x_1,\ldots,x_n\in [0,1]$.
		\end{description}
\end{defi}

An $m$-ary aggregation function $B$ dominates another $m$-ary aggregation function $A$, denoted by $A\leq B$, if $A(x_1,\ldots,x_m)\leq B(x_1,\ldots,x_m)$ for each $x_1,\ldots,x_m\in [0,1]$.

\begin{exem}\label{exem: aggregations}
	Let  $\omega=(w_1,\ldots,w_n)$  be an arbitrary weighting vector, $\vec{e}=(e_1,\ldots,e_n)$ be an  $n$-dimensional vector  of positive real numbers  and $r>0$ be a real number. Then the following functions are examples of $n$-ary aggregation functions.
	\begin{description}
	\item[(a)] $P^r:[0,1]^{n} \rightarrow [0,1]$ defined by
$$P^r(x_1,\ldots,x_n)=\frac{\left (\prod\limits_{i=1}^n (x_i^r+1)\right )- 1}{2^n-1}$$
		\item[(b)] $\min_\omega(x_1, \dots, x_n) = \arg\min(w_ix_i)$ 
		(Weighted min)
		\item[(c)] $\max_\omega(x_1, \dots, x_n) = \arg\max(w_ix_i)$ (Weighted max)
		\item[(d)]  $M_\omega(x_1, \dots, x_n) = \displaystyle\sum_{i = 1}^{n} w_ix_i$ (Weighted average)
		\item[(e)]  $G_\omega(x_1, \dots, x_n) = \displaystyle \prod_{i = 1}^{n} x_i^{w_i}$ (Geometric mean)
		\item[(f)] $\max_{\vec{e}}(x_1,\ldots,x_n)=\max(x_1^{e_1},\ldots,x_n^{e_n})$ (Exponentially distorted maximum)
		\item[(g)] $OWA_\omega(x_1,\ldots,x_n)= \displaystyle\sum_{i = 1}^{n} w_ix_{(i)}$ where $(x_{(1)},\ldots,x_{(n)})$ is a permutation of $(x_1,\ldots,x_n)$ such that $x_{(i)}\geq x_{(i+1)}$ for each $i\in N_1^{n-1}$ (Ordered weighted averaging)
	\end{description}
	In particular, $P^r$ is strict, $\min_\omega$, $\max_\omega$ and  $\max_{\vec{e}}$ are not strict and if $\omega$ is strictly positive then $M_\omega$ and $G_\omega$ are strict. In addition, $P^r$, $G_\omega$ and $\max_{\vec{e}}$ (with some $e_i\not\in\{0,1\}$) are not homogeneous of any order whenever $e_i\neq e_j$ for some $i,j\in \{1,\ldots,n\}$, but $\min_\omega$, $\max_\omega$ and $M_\omega$ are homogeneous of order 1.
	Moreover, we have that $\min_\omega\leq G_\omega\leq M_\omega \leq \max_\omega$.
\end{exem}

%
%

\subsection{Admissible ordered semi-vector spaces on $L_n([0,1])$}

It is easy to note that the algebra  $U=\prt{[0,1],\dotplus,\cdot}$ where $\cdot$ is the usual multiplication and  $\dotplus$ is the bounded addition, i.e. $r\dotplus s=\min(1,r+s)$ for each $r,s\in [0,1]$, is a weak semifield. As usual we will omit the $\cdot$ in expressions like $r\cdot s$ when it is clear.

Given $r\in [0,1]$ and $\mathbf{x},\mathbf{y}\in L_n([0,1])$, we define:
\begin{description}
 \item[Scalar product:] $r\odot \mathbf{x}=(r x_1,\ldots,r x_n)$;
 \item[Bounded addition:] $\mathbf{x}\circplus \mathbf{y}=(x_1\dotplus y_1,\ldots,x_n\dotplus y_n)$.
\end{description}

\begin{teo}\label{teo-SVS}
 $\mathcal{L}_n([0,1])=\prt{L_n([0,1]),\circplus,\odot}$ is a semi-vector space over the weak semifield  $U$. In addition, $\prt{\mathcal{L}_n([0,1]),\leq_n^p}$ is an ordered semi-vector space over $U$.
\end{teo}
\begin{proof}
 First note that the scalar product and the bounded addition are well defined, i.e. $r\odot \mathbf{x},\mathbf{x}\circplus\mathbf{y}\in L_n([0,1])$ for each $r\in [0,1]$ and $\mathbf{x},\mathbf{y}\in L_n([0,1])$.
 Indeed,  for each $i,j\in N_1^n$ such that $i\leq j$,
 $r x_i\leq r x_j$ and $\min(1, x_i+ y_i)\leq \min (1,x_j+y_j)$.

 Since, (SV1) to (SV4) are trivially satisfied by the scalar product and the bounded addition, we will only prove (SV5) - (SV9).

 \begin{enumerate}[start=5,labelindent=\parindent, leftmargin=*,label=\normalfont{(SV\arabic*)}]
  \item $r \odot (\mathbf{x}\circplus \mathbf{y})=(r (x_1\dotplus y_1),\ldots,r (x_n\dotplus y_n))=(r x_1\dotplus r y_1,\ldots,r x_n\dotplus r y_n)= (r x_1,\ldots,r x_n) \circplus (r y_1,\ldots,r y_n) = r \odot \mathbf{x}\dotplus r\odot\mathbf{y}$;
  \item $(r\dotplus s)\odot \mathbf{x}=((r\dotplus s)x_1,\ldots,(r\dotplus s)x_n)=(r x_1\dotplus s x_1,\ldots,r x_n\dotplus s x_n)=(r x_1\ldots,r x_n)\circplus (s x_1,\ldots,s x_n)=r\odot \mathbf{x} \dotplus s\odot \mathbf{x}$;
  \item $\mathbf{x}\circplus /0/=(x_1\dotplus 0,\ldots,x_n\dotplus 0)=\mathbf{x}$;
   \item If $\mathbf{x}\leq_n^p \mathbf{y}$ then $x_i\leq y_i$ for each $i\in N_1^n$ and therefore, because $U$ is a weak semifield,  $r x_i\leq r y_i$ for all $i\in N_1^n$. So, $r\odot \mathbf{x}=(r x_1,\ldots,r x_n)\leq_n^p (r y_1,\ldots,r y_n)=r\odot\mathbf{y}$;
   \item If $\mathbf{x}\leq_n^p \mathbf{y}$ then $x_i\leq y_i$ for each $i\in N_1^n$ and therefore, because $U$ is a weak semifield, $x_i\odot z_i\leq y_i\odot z_i$ for all $i\in N_1^n$. So, $\mathbf{x}\circplus \mathbf{z} =   (x_1\dotplus z_1,\ldots,x_n\dotplus z_n)\leq_n^p  (y_1\dotplus z_1,\ldots,y_n\dotplus z_n)= \mathbf{y}\circplus \mathbf{z}$.
 \end{enumerate}
Therefore, $\prt{\mathcal{L}_n([0,1]),\leq_n^p}$ is an ordered semi-vector space over $U$.
\end{proof}

Observe that the natural preorder (Equation (\ref{eq-Nat-orde-SVS})) of the semi-vector space $\mathcal{L}_n([0,1])$ over the weak semifield $U$, i.e. $\leq_{L_n([0,1])}$,  is in fact, a partial order on $L_n([0,1])$ and consequently, by Proposition \ref{pro-nat-ord}, $\prt{\mathcal{L}_n([0,1]),\leq_{L_n([0,1])}}$ is an ordered semi-vector space over $U$. In addition, $\leq_{L_n([0,1])}$ is  refined by $\leq_n^p$ and therefore by any admissible order. Nevertheless,  not  for all admissible order $\preceq$ the pair $\prt{\mathcal{L}_n([0,1]),\preceq}$ is an ordered semi-vector space over $U$. In fact, consider the aggregation function $A=\max_{\vec{e}}$ of the Example \ref{exem: aggregations} for $\vec{e}=(1,2,3,4)$, $\mathbf{x}=(0.4,0.6,0.7,0.8)$ and $\mathbf{y}=(0.2,0.2,0.2,0.9)$. Then $\max_{\vec{e}}(\mathbf{x})=0.4096$ whereas $\max_{\vec{e}}(\mathbf{y})=0.6561$ and therefore, $\mathbf{x}\prec_\tau^A \mathbf{y}$. However, taking $r=0.5$ we have that  $\max_{\vec{e}}(r\odot\mathbf{x})=\max_{\vec{e}}(0.2,0.3,0.35,0.4)=0.2$ and $\max_{\vec{e}}(r\odot\mathbf{y})=\max_{\vec{e}}(0.1,0.1,0.1,0.45)=0.1$. So,  $r\odot\mathbf{y}\prec_\tau^A  r\odot\mathbf{x}$. Therefore, the pair  $\prt{\mathcal{L}_4([0,1]),\preceq_\tau^A}$ does not satisfy (SV8) and consequently, it is not an ordered semi-vector space over $U$.  This motivates the investigation of  families of admissible orders which are compatible with the  semi-vector space $\mathcal{L}_n([0,1])$.

\begin{prop}\label{prop-tau-scalar-prod}
 For any bijection $\tau:N_1^n\rightarrow N_1^n$, the pair $\prt{\mathcal{L}_n([0,1]),\preceq_\tau}$ is an ordered semi-vector space over $U$.
\end{prop}
\begin{proof} From Theorem \ref{teo-SVS}, we know that $\mathcal{L}_n([0,1])=\prt{L_n([0,1]),\circplus,\odot}$ is a semi-vector space over the weak semifield  $U$. So it is necessary to prove (SV8) and (SV9).

\begin{enumerate}[start=8,labelindent=\parindent, leftmargin=*,label=\normalfont{(SV\arabic*)}]
 \item Let $\mathbf{x},\mathbf{y}\in L_n([0,1])$ and $r\in [0,1]$ be such that   $\mathbf{x}\preceq_\tau \mathbf{y}$.
 If $r=0$ or $\mathbf{x}= \mathbf{y}$ then, by Theorem \ref{teo-SVS},  $r\odot \mathbf{x}\leq_n^p r\odot \mathbf{y}$ and since $\preceq_\tau$ is admissible, then  $r\odot \mathbf{x}\preceq_\tau r\odot \mathbf{y}$.
 Now, consider the case $r>0$ and $\mathbf{x}\prec_\tau  \mathbf{y}$. Then $\exists k\in N_1^n$ such that $\pi_{\tau(k)}(\mathbf{x})<\pi_{\tau(k)}(\mathbf{y})$ and $\forall  i\in N_1^{k-1}$, $\pi_{\tau(i)}(\mathbf{x})=\pi_{\tau(i)}(\mathbf{y})$. Therefore,
 $\pi_{\tau(k)}(r\odot \mathbf{x}) =r \cdot \pi_{\tau(k)}(\mathbf{x})< r  \cdot \pi_{\tau(k)}(\mathbf{y})=\pi_{\tau(k)}(r\odot \mathbf{y})$ and  $\forall i\in N_1^{k-1}$,  $\pi_{\tau(i)}(r\odot \mathbf{x})=r\cdot \pi_{\tau(i)}(\mathbf{x}) = r\cdot \pi_{\tau(i)}(\mathbf{y})=\pi_{\tau(i)}(r\odot \mathbf{y})$.
  Hence, in both cases, $r\odot\mathbf{x}\preceq_\tau r\odot\mathbf{y}$.

  \item Let $\mathbf{x},\mathbf{y},\mathbf{z}\in L_n([0,1])$ be such that   $\mathbf{x}\preceq_\tau \mathbf{y}$. If  $\mathbf{x}= \mathbf{y}$ then trivially  $\mathbf{x}\circplus \mathbf{z}\preceq_\tau\mathbf{y}\circplus \mathbf{z}$.  So, we just need to consider the case $\mathbf{x}\prec_\tau  \mathbf{y}$. In this case, by definition, $\exists k\in N_1^n$ such that $\pi_{\tau(k)}(\mathbf{x})<\pi_{\tau(k)}(\mathbf{y})$ and $\forall  i\in N_1^{k-1}$, $\pi_{\tau(i)}(\mathbf{x})=\pi_{\tau(i)}(\mathbf{y})$. Therefore,
  $\forall  i\in N_1^{k-1}$, $\pi_{\tau(i)}(\mathbf{x}\circplus \mathbf{z})=\pi_{\tau(i)}(\mathbf{x})\dotplus \pi_{\tau(i)}(\mathbf{z})=\pi_{\tau(i)}(\mathbf{y})\dotplus \pi_{\tau(i)}(\mathbf{z})=\pi_{\tau(i)}(\mathbf{y}\circplus \mathbf{z})$ and $\pi_{\tau(k)}(\mathbf{x}\circplus \mathbf{z})=\pi_{\tau(k)}(\mathbf{x})\dotplus \pi_{\tau(k)}(\mathbf{z})\leq \pi_{\tau(k)}(\mathbf{y})\dotplus \pi_{\tau(k)}(\mathbf{z})=\pi_{\tau(k)}(\mathbf{y}\circplus \mathbf{z})$. If $\pi_{\tau(k)}(\mathbf{x}\circplus \mathbf{z})< \pi_{\tau(k)}(\mathbf{y}\circplus \mathbf{z})$ then $\mathbf{x}\circplus \mathbf{z}\prec_\tau\mathbf{y}\circplus \mathbf{z}$. On the other hand, if  $\pi_{\tau(k)}(\mathbf{x}\circplus \mathbf{z})= \pi_{\tau(k)}(\mathbf{y}\circplus \mathbf{z})$ then $\pi_{\tau(l)}(\mathbf{x}\circplus \mathbf{z})=1= \pi_{\tau(l)}(\mathbf{y}\circplus \mathbf{z})$ for each $l\in N_k^n$ and consequently, $\mathbf{x}\circplus \mathbf{z}= \mathbf{y}\circplus \mathbf{z}$.

\end{enumerate}

\end{proof}

\begin{prop} \label{prop-odot-not-circplus}
 Let $n\geq 2$, $\tau:N_1^n\rightarrow N_1^n$ be a bijection and  $\omega=(w_1,\ldots,w_n)$ be a weighting vector. Then $\preceq_\tau^\omega$ is compatible with the scalar product $\odot$, but it is not compatible with the vector addition $\circplus$.
 \end{prop}
\begin{proof} For each $\mathbf{x},\mathbf{y}\in L_n([0,1])$ and $r\in [0,1]$, we have that

\begin{eqnarray} \label{eq-aux-rF} \nonumber
 r\cdot F_\omega(\mathbf{x},\mathbf{y}) &=& r\cdot \sum\limits_{i=1}^n w_i \max(0,\pi_i(\mathbf{x})-\pi_i(\mathbf{y})) \\ \nonumber
 &=&  \sum\limits_{i=1}^n w_i \max(0, r\cdot(\pi_i(\mathbf{x})-\pi_i(\mathbf{y}))) \\ 
&=& \sum\limits_{i=1}^n w_i \max(0,\pi_i(r \odot \mathbf{x})-\pi_i(r \odot \mathbf{y})) \\ \nonumber
&=& F_\omega(r\odot \mathbf{x},r\odot \mathbf{y})
\end{eqnarray}

Let $\mathbf{x},\mathbf{y}\in L_n([0,1])$ and $r\in [0,1]$ be such that   $\mathbf{x}\preceq_\tau^\omega \mathbf{y}$.
 If $r=0$ or $\mathbf{x}= \mathbf{y}$ then $r\odot \mathbf{x}\leq_n^p r\odot \mathbf{y}$ and since $\preceq_\tau^\omega$ is admissible, it follows that  $r\odot \mathbf{x}\preceq_\tau^\omega s\odot \mathbf{y}$.

  Consider the case $r>0$ and $\mathbf{x}\prec_\tau^\omega  \mathbf{y}$. Then, either $F_\omega(\mathbf{x},\mathbf{y}) < F_\omega(\mathbf{y},\mathbf{x})$, or  $F_\omega(\mathbf{x},\mathbf{y}) = F_\omega(\mathbf{y},\mathbf{x})$ and $\mathbf{x}\prec_\tau  \mathbf{y}$.

  If $F_\omega(\mathbf{x},\mathbf{y}) < F_\omega(\mathbf{y},\mathbf{x})$
  then
  $r\cdot  F_\omega(\mathbf{x},\mathbf{y}) < r\cdot F_\omega(\mathbf{y},\mathbf{x})$ and therefore, by Equation (\ref{eq-aux-rF}), $F_\omega(r\odot\mathbf{x},r\odot\mathbf{y}) < F_\omega(r\odot\mathbf{y},r\odot\mathbf{x})$ and consequently,  $r\odot\mathbf{x}\prec_\tau^\omega  r\odot\mathbf{y}$.

   On the other hand, if $F_\omega(\mathbf{x},\mathbf{y}) = F_\omega(\mathbf{y},\mathbf{x})$ then $r\cdot  F_\omega(\mathbf{x},\mathbf{y}) = r\cdot F_\omega(\mathbf{y},\mathbf{x})$ and therefore, by Equation (\ref{eq-aux-rF}),  
   
   \begin{equation}\label{eq-prova-*} F_\omega(r\odot\mathbf{x},r\odot\mathbf{y}) = F_\omega(r\odot\mathbf{y},r\odot\mathbf{x}).\end{equation} 
   
   In addition, if $\mathbf{x}\prec_\tau  \mathbf{y}$ then, by Proposition \ref{prop-tau-scalar-prod},  
   
   \begin{equation}\label{eq-prova-**} r\odot \mathbf{x} \preceq_\tau r\odot \mathbf{y}.
   \end{equation}

   Hence, from (\ref{eq-prova-*}) and (\ref{eq-prova-**}), we have that $r\odot\mathbf{x}\preceq_\tau^\omega  r\odot\mathbf{y}$ and consequently, the pair $\prt{\mathcal{L}_n([0,1]),\preceq_\tau^\omega}$ satisfies (SV8).

 On the other hand, let $\mathbf{x}=(0.5,0.6,\underbrace{1,\ldots, 1}_{(n-2)-times}),\mathbf{y}=(0.3,0.9,\underbrace{1,\ldots, 1}_{(n-2)-times}),\mathbf{z}=(0.1,0.3,\underbrace{1,\ldots, 1}_{(n-2)-times})\in L_n([0,1])$ and $\omega=(0.4,0.6,\underbrace{0,\ldots, 0}_{(n-2)-times})$. Then $F_\omega(\mathbf{x},\mathbf{y})=0.08 < 0.18 =F_\omega(\mathbf{y},\mathbf{x})$ and so  $\mathbf{x}\prec_\tau^\omega \mathbf{y}$. However, $\mathbf{x}\circplus \mathbf{z}=(0.6,0.9,\underbrace{1,\ldots, 1}_{(n-2)-times})$ and $\mathbf{y}\circplus \mathbf{z}=(0.4,\underbrace{1,\ldots, 1}_{(n-1)-times})$. Therefore,  $F_\omega(\mathbf{y}\circplus \mathbf{z},\mathbf{x}\circplus \mathbf{z})=0.06 < 0.08 =F_\omega(\mathbf{x}\circplus \mathbf{z},\mathbf{y}\circplus \mathbf{z})$, and consequently, $\mathbf{y}\circplus \mathbf{z}\prec_\tau^\omega \mathbf{x}\circplus \mathbf{z}$. Hence,   $\prt{\mathcal{L}_n([0,1]),\preceq_\tau^\omega}$ does not satisfy (SV9).
\end{proof}

 Therefore, as a direct consequence of Proposition \ref{prop-odot-not-circplus}, $\prt{\mathcal{L}_n([0,1]),\preceq_\tau^\omega}$ is not an ordered semi-vector space over $U$.

\begin{prop} \label{prop-odot-not-circplus2}
 Let $n\geq 2$, $\tau:N_1^n\rightarrow N_1^n$ be a bijection and  $A:[0,1]^n\rightarrow [0,1]$ be an  aggregation function.  If $A$ is homogeneous of order $k$  then $\preceq_\tau^A$ is  compatible with the scalar product $\odot$. In addition,  $\preceq_\tau^A$ is compatible with the vector addition $\circplus$ if and only if  $\tau$ is  the identity.
 \end{prop}
\begin{proof} Let $\mathbf{x},\mathbf{y}\in L_n([0,1])$ and $r\in [0,1]$ such that   $\mathbf{x}\preceq_\tau^A \mathbf{y}$.
 If $r=0$ or $\mathbf{x}= \mathbf{y}$ then $r\odot \mathbf{x}\leq_n^p r\odot \mathbf{y}$ and since $\preceq_\tau^A$ is an admissible order, then  $r\odot \mathbf{x}\preceq_\tau^A r\odot \mathbf{y}$.

  Consider the case $r>0$ and $\mathbf{x}\prec_\tau^A  \mathbf{y}$. Then, either $A(\mathbf{x}) < A(\mathbf{y})$ or,  $A(\mathbf{x}) = A(\mathbf{y})$ and $\mathbf{x}\prec_\tau  \mathbf{y}$.

If $A(\mathbf{x}) < A(\mathbf{y})$  then $A(r\odot \mathbf{x}) = r^k \cdot A(\mathbf{x}) < r^k\cdot A(\mathbf{y})=A(r\odot \mathbf{y})$ and therefore $r\odot \mathbf{x}\preceq_\tau^A r\odot \mathbf{y}$.

If $A(\mathbf{x}) = A(\mathbf{y})$ and $\mathbf{x}\prec_\tau  \mathbf{y}$ then $A(r\odot \mathbf{x}) = r^k \cdot A(\mathbf{x}) \leq  r^k\cdot A(\mathbf{y})=A(r\odot \mathbf{y})$. If $A(r\odot \mathbf{x})< A(r\odot \mathbf{y})$ then $r\odot \mathbf{x}\prec_\tau^A r\odot \mathbf{y}$. If $A(r\odot \mathbf{x})= A(r\odot \mathbf{y})$ then,  by Proposition \ref{prop-tau-scalar-prod},  $r \odot \mathbf{x}\prec_\tau  r \odot \mathbf{y}$. So, since $\prec_\tau^A $ is an admissible order, then $r\odot\mathbf{x}\prec_\tau^A  r\odot\mathbf{y}$. Hence, $\preceq_\tau^A$ is  compatible with the scalar product $\odot$.

 In addition, if $\tau\neq Id$ then there is $m\in N_1^{n-1}$ such that $\tau(m)>\tau(m+1)$. Let $\mathbf{x}=(x_1,\ldots,x_n)$ and  $\mathbf{y}=(y_1,\ldots,y_n)$ be such that $y_{\tau(m)}> x_{\tau(m)}>x_{\tau(m+1)}> y_{\tau(m+1)}$ and $x_{\tau(i)}=y_{\tau(i)}$ for all $i< m$. Then, clearly, $\mathbf{x}\prec_\tau \mathbf{y}$ and $x_{\tau(i)}\dotplus (1-x_{\tau(m)})=y_{\tau(i)}\dotplus (1-x_{\tau(m)})$ for each $i\leq m$ and $x_{\tau(m+1)}\dotplus (1-x_{\tau(m)})> y_{\tau(m+1)}\dotplus(1-x_{\tau(m)})$. Therefore, $\mathbf{y}\circplus \mathbf{z}\prec_\tau \mathbf{x}\circplus \mathbf{z}$ for $\mathbf{z}=/1-x_{\tau(m)}/$. Hence, $\preceq_\tau^A$ is  not compatible with the vector addition  $\circplus$. 
 However, if $\tau=Id$ and $\mathbf{x},\mathbf{y},\mathbf{z}\in L_n([0,1])$ are such that $\mathbf{x}\preceq_\tau^A \mathbf{y}$ then there exists $m\in N_1^n$ such that $x_m < y_m$ and $x_i=y_i$ for each $i< m$. Therefore, $x_i\dotplus z_i=y_i\dotplus z_i$ for each $i< m$ and if $z_m< 1-x_m$, $x_k\dotplus z_k < y_k\dotplus z_k$ and therefore $\mathbf{x}\circplus \mathbf{z} \prec_\tau^A \mathbf{y}\circplus \mathbf{z}$. If 
  $z_m\geq  1-x_m$ then $\mathbf{x}\circplus \mathbf{z} = \mathbf{y}\circplus \mathbf{z}$. So, $\preceq_{Id}^A$ is  compatible with the vector addition  $\circplus$.
  
\end{proof}


Therefore, as a direct consequence of Proposition \ref{prop-odot-not-circplus2}, when $A$ is homogeneous of some order, then $\prt{\mathcal{L}_n([0,1]),\preceq_\tau^A}$ is  an ordered semi-vector space over $U$ if and only if $\tau=Id$.

\section{$n$-Dimensional Aggregation Functions with respect to an Admissible Order}

Bedregal et al. in \cite{Benja2}  extended the notion of aggregation function for $L_n([0,1])$ taking into account the order $\leq_n^p$. Here we consider an arbitrary admissible order on $L_n([0,1])$.

\begin{defi}
 Let $\preceq$ be an admissible order on $L_n([0,1])$. A function $\mathcal{A}:L_n([0,1])^m \rightarrow L_n([0,1])$ is an $m$-ary $n$-dimensional aggregation function with respect to $\preceq$ if:
 \begin{enumerate}[labelindent=\parindent,leftmargin=*,label=\normalfont{($\mathcal{A}$\arabic*)}]
  \item $\mathcal{A}(/0/,\ldots,/0/)=/0/$ and $\mathcal{A}(/1/,\ldots,/1/)=/1/$; and
  \item  $\mathcal{A}(\mathbf{x}_1,\ldots,\mathbf{x}_m)\preceq \mathcal{A}(\mathbf{y}_1,\ldots,\mathbf{y}_m)$ whenever $\mathbf{x}_j\preceq \mathbf{y}_j$ for each $j\in N_1^m$.
  \end{enumerate}
\end{defi}

\begin{exem}
  Let $\preceq$ be an admissible order on $L_n([0,1])$.  Then $\bigcurlywedge\limits_{j=1}^m \mathbf{x}_j$ and $\bigcurlyvee\limits_{j=1}^m \mathbf{x}_j$ are  $m$-ary $n$-dimensional internal aggregation functions with respect to $\preceq$.
\end{exem}

\begin{defi}
 Let $\preceq$ be an admissible order on $L_n([0,1])$ and let $\mathcal{A}$ be an $m$-ary $n$-dimensional aggregation function with respect to $\preceq$. Then $\mathcal{A}$
 \begin{enumerate}
  \item is conjunctive if $\mathcal{A}(\mathbf{x}_1,\ldots,\mathbf{x}_m)\preceq  \bigcurlywedge\limits_{j=1}^m \mathbf{x}_j$ for each $\mathbf{x}_j\in L_n([0,1])$;
  \item is disjunctive if $\bigcurlyvee\limits_{j=1}^m \mathbf{x}_j\preceq \mathcal{A}(\mathbf{x}_1,\ldots,\mathbf{x}_m)$ for each $\mathbf{x}_j\in L_n([0,1])$;
  \item is an average  if $\bigcurlywedge\limits_{j=1}^m \mathbf{x}_j\preceq \mathcal{A}(\mathbf{x}_1,\ldots,\mathbf{x}_m)\preceq  \bigcurlyvee\limits_{j=1}^m \mathbf{x}_j$ for each $\mathbf{x}_j\in L_n([0,1])$;
  \item is  mixed  if it is neither conjunctive, nor disjunctive nor average;
  \item is idempotent if $\mathcal{A}(\mathbf{x},\ldots,\mathbf{x})=\mathbf{x}$ for each $\mathbf{x}\in L_n([0,1])$;
  \item is strict if for each $\mathbf{x}_1,\ldots,\mathbf{x}_m,\mathbf{y}\in L_n([0,1])$ and $j\in N_1^m$ such that $\mathbf{x}
_j\prec \mathbf{y}$ we have that  $\mathcal{A}(\mathbf{x}_1,\ldots,\mathbf{x}_m)\prec \mathcal{A}(\mathbf{x}_1,\ldots,\mathbf{x}_{j-1},\mathbf{y},\mathbf{x}_{j+1},\ldots,\mathbf{x}_m)$;
\item is internal if for each $\mathbf{x}_1,\ldots,\mathbf{x}_m\in L_n([0,1])$ there exists $j\in N_1^m$ such that $\mathcal{A}(\mathbf{x}_1,\ldots,\mathbf{x}_m)=\mathbf{x}_j$;
\item is symmetric  if for each $\mathbf{x}_1,\ldots,\mathbf{x}_m\in L_n([0,1])$  and for each bijection $\tau: N_1^m\rightarrow N_1^m$ we have that $\mathcal{A}(\mathbf{x}_1,\ldots,\mathbf{x}_m)=\mathcal{A}(\mathbf{x}_{\tau(1)},\ldots,\mathbf{x}_{\tau(m)})$.
 \end{enumerate}
 \end{defi}

 \begin{prop}\label{prop-idempotent-average}
   Let $\preceq$ be an admissible order on $L_n([0,1])$ and  an $m$-ary $n$-dimensional aggregation function $\mathcal{A}$ with respect to $\preceq$. Then $\mathcal{A}$ is an average if and only if $\mathcal{A}$ is idempotent, i.e.  $\mathcal{A}(\mathbf{x},\ldots,\mathbf{x})=\mathbf{x}$ for each $\mathbf{x}\in L_n([0,1])$.
 \end{prop}
 \begin{proof}
  ($\Rightarrow$)  Since $\mathcal{A}$ is an average, then for each $\mathbf{x}\in L_n([0,1])$ we have that $\mathbf{x}=\bigcurlywedge\limits_{j=1}^m \mathbf{x}\preceq\mathcal{A}(\mathbf{x},\ldots,\mathbf{x})\preceq \bigcurlyvee\limits_{j=1}^m \mathbf{x} =\mathbf{x}$ and therefore $\mathcal{A}$ is idempotent.

  ($\Leftarrow$) For each $\mathbf{x}_1,\ldots,\mathbf{x}_n\in L_n([0,1])$ consider $\mathbf{y}=\bigcurlywedge\limits_{j=1}^m \mathbf{x}_j$ and $\mathbf{z}=\bigcurlyvee\limits_{j=1}^m \mathbf{x}_j$. So, since $\mathcal{A}$ is idempotent and increasing with respect to $\preceq$, then

  $\bigcurlywedge\limits_{j=1}^m \mathbf{x}_j=\mathcal{A}(\mathbf{y},\ldots,\mathbf{y})\preceq\mathcal{A}(\mathbf{x}_1,\ldots,\mathbf{x}_m)\preceq \mathcal{A}(\mathbf{z},\ldots,\mathbf{z})=\bigcurlyvee\limits_{j=1}^m \mathbf{x}$ and therefore $\mathcal{A}$ is an average.
 \end{proof}

\begin{prop} \label{prop-A1toAn-tau}
 Let  $m$ and $n$ be non-zero natural numbers, $\tau:N_1^n\rightarrow N_1^n$ be a bijection and $A_1,\ldots,A_n$ be $m$-ary  aggregation functions such that $A_1\leq \cdots \leq A_n$ and they are strictly increasing when restricted to $L_n([0,1])$.
 Then $\widetilde{A_1\ldots A_n}:L_n([0,1])^m\rightarrow L_n([0,1])$ defined by
 $$\widetilde{A_1\ldots A_n}(\mathbf{x}_1,\ldots,\mathbf{x}_m)=(A_1(\pi_1(\mathbf{x}_1),\ldots,\pi_1(\mathbf{x}_m)),\ldots, A_n(\pi_n(\mathbf{x}_1),\ldots,\pi_n(\mathbf{x}_m)))$$
 is an  $m$-ary strict $n$-dimensional aggregation function with respect to $\preceq_{\tau}$. In addition, if $A_i$'s are conjunctive (disjunctive, internal) $m$-ary aggregation functions  for each $i\in N_1^n$ then $\widetilde{A_1\ldots A_n}$ is a conjunctive (disjunctive, average) $m$-ary $n$-dimensional aggregation function with respect to $\preceq_{\tau}$.
\end{prop}
\begin{proof} Clearly, $\widetilde{A_1\ldots A_n}$  is well defined,
  $\widetilde{A_1\ldots A_n}(/0/,\ldots,/0/)=/0/$ and $\widetilde{A_1\ldots A_n}(/1/,\ldots,/1/)=/1/$. Let $j\in N_m$ and $\mathbf{x}_1,\ldots,\mathbf{x}_m,\mathbf{y}\in L_n([0,1])$ such that $\mathbf{x}_j\prec_{\tau} \mathbf{y}$. Then, $\exists k\in N_1^n\mbox{ s.t. }\pi_{\tau(k)}(\mathbf{x}_j)<\pi_{\tau(k)}(\mathbf{y})\mbox{ and }\forall i, 1\leq i < k, \pi_{\tau(i)}(\mathbf{x}_j)=\pi_{\tau(i)}(\mathbf{y})$. So, for each $1\leq i < k$,   we have that
 $$\begin{array}{ll}
    \pi_{\tau(i)}(\widetilde{A_1\ldots A_n}(\mathbf{x}_1,\ldots,\mathbf{x}_m)) & =A_{\tau(i)}(\pi_{\tau(i)}(\mathbf{x}_1),\ldots,\pi_{\tau(i)}(\mathbf{x}_m)) \\
    & = A_{\tau(i)}(\pi_{\tau(i)}(\mathbf{x}_1),\ldots,\pi_{\tau(i)}(\mathbf{x}_{j-1}), \pi_{\tau(i)}(\mathbf{y}),\pi_{\tau(i)}(\mathbf{x}_{j+1}),\ldots,\pi_{\tau(i)}(\mathbf{x}_m)) \\
    & =  \pi_{\tau(i)}(\widetilde{A_1\ldots A_n}(\mathbf{x}_1,\ldots,\mathbf{x}_{j-1}, \mathbf{y}, \mathbf{x}_{j+1},\ldots,\mathbf{x}_m))
   \end{array}
$$ and, because $A_{\tau(k)}$ when restricted to $L_n([0,1])$ is strictly increasing, we have that
 $$\begin{array}{ll}
    \pi_{\tau(k)}(\widetilde{A_1\ldots A_n}(\mathbf{x}_1,\ldots,\mathbf{x}_m)) & = A_{\tau(k)}(\pi_{\tau(k)}(\mathbf{x}_1),\ldots,\pi_{\tau(k)}(\mathbf{x}_m)) \\
    & < A_{\tau(k)}(\pi_{\tau(k)}(\mathbf{x}_1),\ldots,\pi_{\tau(k)}(\mathbf{x}_{j-1}), \pi_{\tau(k)}(\mathbf{y}),\pi_{\tau(k)}(\mathbf{x}_{j+1}),\ldots,\pi_{\tau(k)}(\mathbf{x}_m)) \\
    & =  \pi_{\tau(k)}(\widetilde{A_1\ldots A_n}(\mathbf{x}_1,\ldots\mathbf{x}_{j-1}, \mathbf{y}, \mathbf{x}_{j+1},\mathbf{x}_m))
    \end{array}$$
 Therefore,
 $$\widetilde{A_1\ldots A_n}(\mathbf{x}_1,\ldots,\mathbf{x}_m)\prec_{\tau}
 \widetilde{A_1\ldots A_n}(\mathbf{x}_1,\ldots,\mathbf{x}_{j-1}, \mathbf{y}, \mathbf{x}_{j+1},\ldots,\mathbf{x}_m).
$$
Consequently, $\widetilde{A_1\ldots A_n}$ is a $m$-ary strict $n$-dimensional aggregation function with respect to $\preceq_{\tau}$.

Now, suppose that all the $m$-dimensional aggregation functions $A_i$ are conjunctive. Then $A_i(\pi_i(\mathbf{x}_i),\ldots,\pi_i(\mathbf{x}_m))\leq \min (\pi_i(\mathbf{x}_i),\ldots,\pi_i(\mathbf{x}_m))$ for each $i\in N_1^n$. Therefore, $\widetilde{A_1\ldots A_n}(\mathbf{x}_1,\ldots,\mathbf{x}_m)\leq_n^p \bigwedge \{\mathbf{x}_1,\ldots,\mathbf{x}_m\}\leq_n^p \mathbf{x}_j$ for each $j\in N_1^m$. So, because, $\preceq_\tau$ is an admissible order, $\widetilde{A_1\ldots A_n}(\mathbf{x}_1,\ldots,\mathbf{x}_m)\preceq_\tau \bigwedge \{\mathbf{x}_1,\ldots,\mathbf{x}_m\}\preceq_\tau \bigcurlywedge\limits_{j=1}^m \mathbf{x}_j$. Analogously, we can prove that $\bigcurlyvee\limits_{j=1}^m \mathbf{x}_j\preceq_\tau \bigvee \{\mathbf{x}_1,\ldots,\mathbf{x}_m\}\preceq_\tau \widetilde{A_1\ldots A_n}(\mathbf{x}_1,\ldots,\mathbf{x}_m)$ whenever $A_j$ is disjunctive.

If the $A_i$'s are internal aggregation functions then $A_{i}(\pi_i(\mathbf{x}_1),\ldots,\pi_i(\mathbf{x}_m))=\pi_i(\mathbf{x}_{j_i})$ for some $j_i\in N_1^m$. Let $t=\min\limits_{i=1}^m j_i$ and $s=\max\limits_{i=1}^m j_i$. Then $\pi_i(\mathbf{x}_t)\leq A_i(\pi_i(\mathbf{x}_1),\ldots, \pi_i(\mathbf{x}_m))\leq \pi_i(\mathbf{x}_s)$ and therefore $\mathbf{x}_t\leq_n^p \widetilde{A_1\ldots A_n}(\mathbf{x}_1,\ldots,\mathbf{x}_m)\leq_n^p \mathbf{x}_s$ So, because $\preceq_\tau$ is an admissible order,

$$\bigcurlywedge\limits_{j=1}^m \mathbf{x}_j\preceq_\tau \mathbf{x}_t\preceq_\tau\widetilde{A_1\ldots A_n}(\mathbf{x}_1,\ldots,\mathbf{x}_m)\preceq_\tau\mathbf{x}_s\preceq_\tau\bigcurlyvee\limits_{j=1}^m \mathbf{x}_j.$$
Hence, $\widetilde{A_1\ldots A_n}$ is an average $m$-ary $n$-dimensional aggregation function with respect to $\preceq_{\tau}$.
\end{proof}

\begin{obs}
 In  Proposition \ref{prop-A1toAn-tau}, the condition that all $A_i$ restricted to $L_n([0,1])$ must be  strictly increascing  is necessary because if some $A_i$ does not satisfy this condition  then $\widetilde{A_1\ldots A_n}$ may not  be a $m$-ary $n$-dimensional aggregation function with respect to $\preceq_{\tau}$. Indeed, taking as $\tau$ the identity, $A_1(x_1,\ldots,x_m)=\ldots =A_{n-1}=\min(x_1,\ldots,x_m)$ and  $A_n(x_1,\ldots,x_m)=\frac{1}{m}\sum\limits_{i=1}^m x_i$ then $\mathbf{x}=(0.5,\ldots,0.5,1)\prec_\tau (0.5,\ldots,0.5,0.6,0.6)=\mathbf{y}$. However, \\
 $$\widetilde{A_1\ldots A_n}(\mathbf{x},\ldots,\mathbf{x},\mathbf{y})= (0.5,\ldots,0.5,0.5+\frac{0.1}{m})\prec_\tau (0.5,\ldots,0.5,\frac{m+1}{2m})=\widetilde{A_1\ldots A_n}(\mathbf{x},\ldots,\mathbf{x}).$$
\end{obs}


\begin{teo}\label{teo-OWA-Ln-preceq} Let   $\preceq$ be an admissible order on $L_n([0,1])$, $\omega=(w_1,\ldots,w_m)$ be a weighting vector and $\mathcal{L}'_n([0,1])=\prt{L_n([0,1]),+,\star,\preceq}$ be an ordered semi-vector space  over $U$.
 Then the  function $OWA_\omega^{\mathcal{L}'_n([0,1])}: L_n([0,1])^m\rightarrow L_n([0,1])$ defined by
\begin{equation}\label{eq-OWA-preceq}
 OWA_\omega^{\mathcal{L}'_n([0,1])}(\mathbf{x}_1,\ldots,\mathbf{x}_m)=\sum\limits_{j=1}^m w_j\star \mathbf{x}_{(j)}
\end{equation}
where  the sum is with respect to $+$, $\mathbf{x}_{(j+1)}\preceq \mathbf{x}_{(j)}$ for each $j\in N_1^{m-1}$ and $\{(1),\ldots,(m)\}=\{1,\ldots,m\}$, is an idempotent $m$-ary $n$-dimensional  aggregation function with respect to $\preceq$ called $m$-ary $n$-dimensional  ordered weighted aggregation function with respect to $\preceq$.
 \end{teo}
\begin{proof}
For each $\mathbf{x} \in L_n([0,1])$ we have that,

 $$\begin{array}{ll}
    OWA_\omega^{\mathcal{L}'_n([0,1])}(\mathbf{x},\ldots,\mathbf{x}) & = \sum\limits_{j=1}^m w_j\star \mathbf{x}  \,\,\mbox{, by Eq. (\ref{eq-OWA-preceq})} \\
    & =  \left(\sum\limits_{j=1}^m w_j\right ) \star \mathbf{x} \,\, \mbox{, by (SV6)} \\
    & = 1 \boxdot \mathbf{x} \\
    & = \mathbf{x} \,\, \mbox{, by (SV4)}
   \end{array}
   $$
   i.e. $OWA_\omega^{\mathcal{L}'_n([0,1])}$ is idempotent and therefore, $OWA_\omega^{\mathcal{L}'_n([0,1])}(/0/_n,\ldots,/0/_n)=/0/_n$ and $OWA_\omega^{\mathcal{L}'_n([0,1])}(/1/_n,\ldots,/1/_n)=/1/_n$.

   Let $\mathbf{x}_1,\ldots,\mathbf{x}_m,\mathbf{y}\in L_n([0,1])$ and $t\in N_1^m$.
   If $\mathbf{x}_{(t)}\preceq \mathbf{y}$ then by (SV8) we have that $w_t\star \mathbf{x}_{(t)}\preceq w_t \star \mathbf{y}$ and therefore,

   $$\begin{array}{ll}
    OWA_\omega^{\mathcal{L}'_n([0,1])}(\mathbf{x}_1,\ldots,\mathbf{x}_m) & = \sum\limits_{j=1}^m w_j\star \mathbf{x}_{(j)}  \,\,\mbox{, by Eq. (\ref{eq-OWA-preceq})} \\
    &= \left ( \sum\limits_{j=1, j\neq t}^m w_j\star \mathbf{x}_{(j)} \right ) + (w_t\star \mathbf{x}_{(t)}) \,\, \mbox{, by (SV1) and (SV2)} \\
    & \preceq \left ( \sum\limits_{j=1, j\neq t}^m w_j\star \mathbf{x}_{(j)} \right ) + (w_t\star \mathbf{y}) \,\, \mbox{, by (SV9)} \\
    & =  OWA_\omega^{\mathcal{L}'_n([0,1])}(\mathbf{x}_1,\ldots,\mathbf{x}_{(t)-1},\mathbf{y},\mathbf{x}_{(t)},\ldots,\mathbf{x}_m) \,\,\mbox{, by Eq. (\ref{eq-OWA-preceq})}
    \end{array}
   $$
   Therefore, $OWA_\omega^{\mathcal{L}'_n([0,1])}$ is increasing with respect to $\preceq$ and consequently is an idempotent  $m$-ary $n$-dimensional  aggregation function with respect to $\preceq$.
\end{proof}

\begin{coro}
Let   $\preceq$ be an admissible order on $L_n([0,1])$,  $\omega=(w_1,\ldots,w_m)$ be a weighting vector and $\mathcal{L}'_n([0,1])=\prt{L_n([0,1]),+,\star,\preceq}$ be an ordered semi-vector space  over $U$.
 Then $OWA_\omega^{\mathcal{L}'_n([0,1])}$ is a  $m$-ary $n$-dimensional  average aggregation function with respect to $\preceq$.
\end{coro}
\begin{proof}
 Straightforward from Proposition \ref{prop-idempotent-average} and Theorem \ref{teo-OWA-Ln-preceq}.
\end{proof}

\begin{coro}
Let   $\tau:N_1^m\rightarrow N_1^m$ be a bijection  and  $\omega=(w_1,\ldots,w_m)$ be a weighting vector.
 Then $OWA_\omega^{\mathcal{L}_n([0,1])}$ is a  $m$-ary $n$-dimensional  idempotent and average aggregation function with respect to $\preceq_\tau$.
\end{coro}
\begin{proof}
 Straightforward from Propositions \ref{prop-tau-scalar-prod} and  \ref{prop-idempotent-average} and Theorem \ref{teo-OWA-Ln-preceq}.
\end{proof}

\begin{teo}\label{teo-Mw-Ln-preceq}
  Let   $\preceq$ be an admissible order on $L_n([0,1])$, $\omega=(w_1,\ldots,w_m)$ be a weighting vector and $\mathcal{L}'_n([0,1])=\prt{L_n([0,1]),+,\star,\preceq}$ be an ordered semi-vector space  over $U$.
 Then the  function  $\mathcal{M}_\omega^{\mathcal{L}'_n([0,1])}: L_n([0,1])^m\rightarrow L_n([0,1])$ defined by
\begin{equation}\label{eq-Mw-preceq}
 \mathcal{M}_{\omega}^{\mathcal{L}'_n([0,1])}(\mathbf{x}_1,\ldots,\mathbf{x}_m)=\sum\limits_{j=1}^m w_j\star \mathbf{x}_{j}
\end{equation}
where the sum is with respect to $+$, is an idempotent $m$-ary $n$-dimensional  aggregation function with respect to $\preceq$ called of $m$-ary $n$-dimensional weighted average with respect to $\preceq$.
\end{teo}
\begin{proof} Analogous to the proof of Theorem \ref{teo-OWA-Ln-preceq}.
\end{proof}

\begin{coro}
Let   $\preceq$ be an admissible order on $L_n([0,1])$,  $\omega=(w_1,\ldots,w_m)$ be a weighting vector and $\mathcal{L}'_n([0,1])=\prt{L_n([0,1]),+,\star,\preceq}$ be an ordered semi-vector space  over $U$.
 Then $M_\omega^{\mathcal{L}'_n([0,1])}$ is a  $m$-ary $n$-dimensional  average aggregation function with respect to $\preceq$.
\end{coro}
\begin{proof}
 Straightforward from Proposition \ref{prop-idempotent-average} and Theorem \ref{teo-Mw-Ln-preceq}.
\end{proof}

\begin{coro}
Let   $\tau:N_1^m\rightarrow N_1^m$ be a bijection  and  $\omega=(w_1,\ldots,w_m)$ be a weighting vector.
 Then $M_\omega^{\mathcal{L}_n([0,1])}$ is a  $m$-ary $n$-dimensional  idempotent and average aggregation function with respect to $\preceq_\tau$.
\end{coro}
\begin{proof}
 Straightforward from Propositions \ref{prop-tau-scalar-prod} and  \ref{prop-idempotent-average} and Theorem \ref{teo-Mw-Ln-preceq}.
\end{proof}

In the following, we will explore some properties of $M_\omega^{\mathcal{L}_n([0,1])}$.

\begin{lem}\label{lem-preceq-tau-strict}
 Let $\tau:N_1^m\rightarrow N_1^m$ be a bijection, $\lambda\in (0,1)$ and  $\mathbf{x},\mathbf{y},\mathbf{z}\in L_n([0,1])$ such that $\mathbf{x}\prec_\tau \mathbf{y}$ and $\pi_n(\mathbf{z}) \leq 1-\lambda$. Then $\mathbf{z}\circplus \lambda\odot \mathbf{x}\prec_\tau \mathbf{z}\circplus \lambda\odot \mathbf{y}$.
\end{lem}
\begin{proof}
 Since $\mathbf{x}\prec_\tau \mathbf{y}$ then there exists $k\in N_1^n$ such that $\pi_k(\mathbf{x})<\pi_k(\mathbf{y})$ and $\pi_i(\mathbf{x})=\pi_i(\mathbf{y})$ for each $i <k$. So, because $\lambda> 0$, then
  $\lambda \pi_k(\mathbf{x})<\lambda \pi_k(\mathbf{y})$ and $\lambda \pi_i(\mathbf{x})=\lambda \pi_i(\mathbf{y})$ for each $i <k$,  and therefore,
 $\lambda\odot \mathbf{x}\prec_\tau \lambda\odot \mathbf{y}$. So, once $\lambda<1$ and $\pi_k(\mathbf{z}) \leq 1-\lambda$, we have that
 $$\begin{array}{ll}
   \pi_k(\mathbf{z})\dotplus \pi_k(\lambda\odot \mathbf{x}) & =\pi_k(\mathbf{z})\dotplus \lambda\pi_k(\mathbf{x}) \\
   & =\pi_k(\mathbf{z})+ \lambda\pi_k(\mathbf{x}) \\
   & < \pi_k(\mathbf{z})+ \lambda\pi_k(\mathbf{y}) \\
  & =\pi_k(\mathbf{z})\dotplus \lambda\pi_k(\mathbf{y}) \\
   & = \pi_k(\mathbf{z})\dotplus \pi_k(\lambda\odot \mathbf{y})
  \end{array}
  $$
  Thus, since clearly,  $\pi_i(\mathbf{z})\dotplus \pi_i(\lambda\odot \mathbf{x}) =
  \pi_i(\mathbf{z})\dotplus \pi_i(\lambda\odot \mathbf{y})$ for each $i<k$ then $\mathbf{z}\circplus \lambda\odot \mathbf{x}\prec_\tau \mathbf{z}\circplus \lambda\odot \mathbf{y}$.
\end{proof}

\begin{prop}
 Let   $\tau:N_1^m\rightarrow N_1^m$ be a bijection  and  $\omega=(w_1,\ldots,w_m)$ be a weighting vector. Then $M_\omega^{\mathcal{L}_n([0,1])}$ is
 \begin{enumerate}
 \item strict if and only if $w_j>0$ for each $j\in N_1^m$;
  \item symmetric if and only if $w_j=\frac{1}{m}$ for each $j\in N_1^m$;
  \item additive with respect to $\circplus$, i.e. $M_\omega^{\mathcal{L}_n([0,1])}(\mathbf{x}_1\circplus \mathbf{y}_1,\ldots,\mathbf{x}_m\circplus \mathbf{y}_m)= M_\omega^{\mathcal{L}_n([0,1])}(\mathbf{x}_1,\ldots,\mathbf{x}_m)\circplus M_\omega^{\mathcal{L}_n([0,1])}(\mathbf{y}_1,\ldots,\mathbf{y}_m)$ for each $\mathbf{x}_1,\mathbf{y}_1,\ldots,\mathbf{x}_m,\mathbf{y}_m\in L_n([0,1])$;
  \item homogeneous, i.e. $M_\omega^{\mathcal{L}_n([0,1])}(\lambda\odot \mathbf{x}_1,\ldots,\lambda\odot \mathbf{x}_m)=\lambda\odot M_\omega^{\mathcal{L}_n([0,1])}(\mathbf{x}_1,\ldots,\mathbf{x}_m)$ for each $\lambda\in [0,1]$ and $\mathbf{x}_1,\ldots,\mathbf{x}_m\in L_n([0,1])$.
 \end{enumerate}
\end{prop}
\begin{proof}
  \begin{enumerate}
 \item ($\Rightarrow$) For each $j\in N_1^m$,

 $$\begin{array}{ll}
/0/ & =M_\omega^{\mathcal{L}_n([0,1])}(/0/,\ldots,/0/) \\
& \prec_\tau M_\omega^{\mathcal{L}_n([0,1])}(\underbrace{/0/,\ldots,/0/}_{(j-1)-times},/1/,\underbrace{/0/,\ldots,/0/}_{(m-j)-times}) \\
& =w_j\odot /1/ \\
& =/w_j/
   \end{array}$$
   and therefore $0< w_j$ for each $j\in N_1^m$.

   ($\Leftarrow$)  Let $\mathbf{x}_1,\ldots,\mathbf{x}_m,\mathbf{y}\in L_n([0,1])$. Since $0< w_j$ for each $j\in N_1^m$ then for each $\mathbf{x},\mathbf{y}\in L_n([0,1])$ we have that $w_j \odot \mathbf{x} \prec_\tau w_j \odot \mathbf{y}$ whenever $\mathbf{x} \prec_\tau \mathbf{y}$. So,  for each $t\in  N_1^m$ such that $\mathbf{x}_t\prec_\tau \mathbf{y}$, we have that

 $$\begin{array}{ll}
   \mathcal{M}_{\omega}^{\mathcal{L}_n([0,1])}(\mathbf{x}_1,\ldots,\mathbf{x}_m)  & =\sum\limits_{j=1}^m w_j\odot \mathbf{x}_{j} \\
   & = \left (\sum\limits_{j=1,j\neq t}^m w_j\odot \mathbf{x}_{j} \right ) \circplus w_t\odot \mathbf{x}_{t} \\
   & \prec_\tau    \left (\sum\limits_{j=1,j\neq t}^m w_j\odot \mathbf{x}_{j} \right ) \circplus w_t\odot \mathbf{y} \,\, \mbox{  (by Lemma \ref{lem-preceq-tau-strict})} \\
   & = \mathcal{M}_{\omega}^{\mathcal{L}_n([0,1])}(\mathbf{x}_1,\ldots,\mathbf{x}_{t-1},\mathbf{y},\mathbf{x}_{t+1},\ldots,\mathbf{x}_m)
  \end{array}
 $$

 \item ($\Rightarrow$) Let $j\in N_1^m$  then
 $$\begin{array}{ll}
   /w_j/& =\mathcal{M}_{\omega}^{\mathcal{L}_n([0,1])}(\underbrace{/0/,\ldots,/0/}_{(j-1)-times},/1/,\underbrace{/0/,\ldots,/0/}_{(m-j)-times}) \\
   & =\mathcal{M}_{\omega}^{\mathcal{L}_n([0,1])}(\underbrace{/0/,\ldots,/0/}_{(\tau(j)-1)-times},/1/,\underbrace{/0/,\ldots,/0/}_{(m-\tau(j))-times}) \\
   & =/w_{\tau(j)}/.
  \end{array}
$$
Therefore, $w_j=w_{\tau(j)}$ for each bijection $\tau$ and, consequently, $w_j=\frac{1}{m}$ for each $j\in N_1^m$.

 ($\Leftarrow$) Straightforward.

 \item Let $\mathbf{x}_1,\ldots,\mathbf{x}_m,\mathbf{y}_1,\ldots,\mathbf{y}_m\in L_n([0,1])$. Then

 $$\begin{array}{ll}
    M_\omega^{\mathcal{L}_n([0,1])}(\mathbf{x}_1{\circplus} \mathbf{y}_1,\ldots,\mathbf{x}_m{\circplus} \mathbf{y}_m) & = \sum\limits_{j=1}^m w_j\odot (\mathbf{x}_j\circplus \mathbf{y}_j) \\
    & = \sum\limits_{j=1}^m (w_j (\pi_1(\mathbf{x}_j)\dotplus \pi_1(\mathbf{y}_j)),\ldots,w_j(\pi_n(\mathbf{x}_j)\dotplus \pi_n(\mathbf{y}_j) ))\\
    & = \sum\limits_{j=1}^m (w_j \pi_1(\mathbf{x}_j)\dotplus w_j \pi_1(\mathbf{y}_j),\ldots,w_j \pi_n(\mathbf{x}_j)\dotplus w_j\pi_n(\mathbf{y}_j)) \\
    & = \sum\limits_{j=1}^m ((w_j \pi_1(\mathbf{x}_j),\ldots,w_j \pi_n(\mathbf{x}_j)) \dotplus (w_j \pi_1(\mathbf{y}_j),\ldots, w_j\pi_n(\mathbf{y}_j)) ) \\
    & =  \left( \sum\limits_{j=1}^m w_j\odot \mathbf{x}_j\right) \circplus  \sum\limits_{j=1}^m w_j\odot  \mathbf{y}_j \\
    & = M_\omega^{\mathcal{L}_n([0,1])}(\mathbf{x}_1,\ldots,\mathbf{x}_m)  \circplus M_\omega^{\mathcal{L}_n([0,1])}(\mathbf{y}_1,\ldots,\mathbf{y}_m)
   \end{array}
$$

\item Let  $\mathbf{x}_1,\ldots,\mathbf{x}_m\in L_n([0,1])$ and $\lambda\in [0,1]$. Then
$$\begin{array}{ll}
   M_\omega^{\mathcal{L}_n([0,1])}(\lambda\odot \mathbf{x}_1,\ldots,\lambda\odot\mathbf{x}_m) & = \sum\limits_{j=1}^m w_j\odot (\lambda\odot \mathbf{x}_j)\\
   & = \sum\limits_{j=1}^m \lambda \odot (w_j\odot \mathbf{x}_j) \,\,\mbox{ (by (SV3) and (WF2))}\\
   & = \lambda\odot \sum\limits_{j=1}^m w_j\odot \mathbf{x}_j \,\,\mbox{ (by (SV5))}\\
   & = \lambda\odot M_\omega^{\mathcal{L}_n([0,1])}( \mathbf{x}_1,\ldots,\mathbf{x}_m)
  \end{array}
$$
 \end{enumerate}
\end{proof}

\section{A Decision-Making Method Based on $n$-dimensional Aggregation Functions}

Multiple criteria group decision making (MCGDM) methods are procedures to rank a set of alternatives, by considering some criteria and the opinion of some experts on how much the alternatives satisfy each one of the criteria. In general, the criteria may be in conflict to each other,  which makes more difficult the generation of the ranking in a reasonable way.

Let $A = \{a_1, \ldots, a_p\}$ $(p \geq 2)$ be a set of alternatives, $E = \{e_1, \ldots, e_n\}$ $(n \geq 2)$ be a group of experts, and $C = \{c_1, \ldots, c_m\}$ $(m \geq 2)$ be a set of attributes or criteria. From the opinion of each expert $e_k$, we generate a decision matrix $R^k = (R^k_{ij})_{p\times m}$  such that $R^k_{ij}$ is  the evaluation of $e_k$ of how much the alternative $a_i$ meets the criterion $c_j$, which is expressed by a value in $[0,1]$.

We will consider the immersion $\sigma:[0,1]^n\rightarrow L_n([0,1])$, i.e. the function
 \begin{equation}\label{eq-sigma}
  \sigma(x_1,\ldots,x_n)=(x_{(1)},\ldots,x_{(n)})
 \end{equation}
such that $(x_{(1)},\ldots,x_{(n)})$ is the element of $L_n([0,1])$  which is a permutation of $(x_1,\ldots,x_n)$. Observe that $(x_{(1)},\ldots,x_{(n)})$  always exists and is unique. For example, $\sigma(0.2,0.5,0.2,0.4,0.5,0.7,0.5)=(0.2,0.2,0.4,0.5,0.5,0.5,0.7)$. However, as this example shows, the corresponding permutation may no be unique.

We propose the following general multi-criteria and multi-expert decision-making method:

\begin{enumerate}[label=\textbf{Step \arabic*.}]
 \item For every  expert $e_k\in E$, generate the matrix $R^k$ with his/her evaluations of how much each alternative meets each criterion

\item  Generate the collective matrix $R^*$ from the   $R^k$'s with $k=1,\ldots,n$,  by ordering it, so that $R^*_{ij}$  belongs to $L_n([0,1])$ for each $i\in N_1^p$ and $j\in N_1^m$, i.e., $R^*_{ij}=\sigma(R^1_{ij},\ldots,R^n_{ij})$.

\item  Select a strictly positive weighting vector $\omega=(w_1,\ldots,w_m)$  which reflects the importance of each criterium $c_j$ in the decision making process.

\item Choose an admissible order $\preceq$ on $L_n([0,1])$ and a commutative  $m$-ary $n$-dimensional aggregation function $\mathcal{A}$, with respect to $\preceq$.

\item  Apply  $\mathcal{A}$  to each row of $R^*$, considering \textbf{Step 3}, i.e. $s_i=\mathcal{A}(R^*_{i1},\ldots,R^*_{im})$ is the $L_n([0,1])$-score of the  alternative $a_i$.

\item  Determine the ranking of the alternatives in $A$ from the $L_n([0,1])$-scores obtained in the \textbf{Step 5}, taking into account the admissible order $\preceq$ choice in the \textbf{Step 4}.

\end{enumerate}

 Many MCGDM methods had been proposed but most of them do not  consider any study of their general properties and therefore, we have no guarantee or evidence that the resulting ranking is reasonable. For example, the first illustrative example in \cite{MC11} presents 12 different rankings obtained from the proposed method when considering 12 different aggregation function, and in  \cite{dSBB15,dSBC15} for that same illustrative example, 3 new different rankings were obtained. That is, at all, we get 15 different ranking for the same problem and we have no mathematical foundation to determine which of these ranking is best one. Nevertheless, it is possible to determine some good properties for decision making methods, as made in \cite{BMP2009} for decision making methods based on preference relations and  in the light of Arrow's social choice theory (see \cite{Arrow63,Mord15}).

In the following, we present two reasonable principles that any MCGDM methods based on decision matrices should  satisfy, in our opinion:

\begin{description}
 \item {\bf Increasingness:} If the value of some $R_{ij}^k$ is increased then the ranking of $a_i$ not can decrease.
 \item {\bf Domination:} If the alternative $a_i$ is better evaluated  than $a_j$  in all the criteria and in the opinion of all the experts, then  the ranking of $a_i$ must be better than the ranking of $a_j$.
 \item {\bf Insensibility to indexations:} The order (indexation) of the expert, alternatives and criteria is irrelevant, i.e. different orders of the experts must never modify the final ranking of the alternatives.
\end{description}

  The first principle is guaranteed in our method. In fact, when $R_{ij}^k$  is increased to  $R'^k_{ij}$, for some  $i\in N_1^p$, $j\in N_1^m$ and $k\in N_1^n$,
then    $R^*_{ij}\leq_n^p R'^*_{ij}$,  where $R^*$ and $R'^*$ are the collective matrices generated in the step 2 before and after the increasing of  $R_{ij}^k$.
 So, since $\preceq$ is an admissible order and $\mathcal{A}$ is an aggregation function with respect to this admissible order, we have that $s_i\preceq s'_i$, where $s_i$ and $s'_i$ are  the $L_n([0,1])$-scores, computed in the step 5, of the  alternative $a_i$ before and after this increasing, respectively, wehereas the $L_n([0,1])$-scores of the  other alternatives remain the same. Therefore, if $a_j$ is worse ranked than $a_i$ before of the change, then $s_j\preceq s_i\preceq s'_i$. Hence,  $a_j$ will remain worse ranked than $a_i$ after  of the change.

 The domination condition is also satisfied by the proposed method. Indeed, suppose that an alternative $a_i$ is better evaluated  than $a_j$  in all the criteria for all the experts, i.e. $R^k_{jl}\leq R^k_{il}$ for each $k\in N_1^n$ and $l\in N_1^m$. Then, $R^*_{jl}\leq_n^p R^*_{il}$ and therefore $R^*_{jl}\preceq R^*_{il}$. So, because $\mathcal{A}$ is a $m$-ary aggregation function with respect to $\preceq$,  we have that $\mathcal{A}(R^*_{j1},\ldots,R^*_{jm})\preceq \mathcal{A}(R^*_{i1},\ldots,R^*_{im})$ and therefore, $s_j\preceq s_i$ which means that the ranking of the alternative $a_j$ is worse than the ranking of alternative $a_i$.

Finally,  for each permutation $\rho_u$ of $N_1^u$ with $u\in \{m,n,p\}$, define $P^k_{ij}=R^{\rho_n(k)}_{\rho_p(i)\rho_m(j)}$ and $t_i$ the  $L_n([0,1])$-score of the  alternative $b_i=a_{\rho_p(i)}$ obtained by the method. Since $\sigma$ and $\mathcal{A}$ are commutative, then $t_i=s_{\rho_p(i)}$ for each $i\in N_1^p$ and therefore, the ranking of the alternative $a_{\rho_p(i)}$ is the same as the alternative $b_i$, but $b_i=a_{\rho_p(i)}$. Therefore, the ranking is not modified by the permutations and consequently, we can claim that the proposed method also satisfies the unsensibility to indexations principle.


\subsection{An energy police selection problem under a multidimensional fuzzy environment}

As an illustrative example of the proposed method based on $n$-dimensional aggregation functions with respect to an admissible order, we present an energy policy problem adapted from \cite{XuZ}. Since the energy issue has major economic, social and environmental impacts, addressing the right ways to deal with this issue affects economic and environmental development in societies. Therefore, selecting the most appropriate energy policy is very significant. Suppose there are five energy projects, i.e., our set of alternatives is $\{a_1, a_2, a_3, a_4, a_5\}$. Such alternatives are evaluated according to the criteria: technology ($C_1$), environment ($C_2$), socio-political ($C_3$) and economic ($C_4$). We will consider that five experts ($e_1$,$e_2$,$e_3$,$e_4$,$e_5$) provide their decision matrices, i.e. matrices $R^k$ (where $k$ denote that is the decision matrix of expert $e_k$) such that each position $(i,j)$ in $R^k$ contain a value $R_{ij}^k$ in $[0,1]$ corresponding to their belief or evaluation of how much the alternative $a_i$ attends the criteria $C_j$. The decision matrices provided by the five experts are shown in Tables \ref{Table_e1}--\ref{Table_e5} (Step 1).

\begin{table}[h!]
	\begin{minipage}{.30\textwidth}
		\caption{Decision matrix  of the expert $e_1$}
		\label{Table_e1}
		\centering
		\setlength{\tabcolsep}{2pt}
		\begin{tabular}{c|cccc}
			\hline
			 $\scriptstyle R^1$ & $\scriptstyle C_1$ & $\scriptstyle C_2$ & $\scriptstyle C_3$ & $\scriptstyle C_4$ \\ \hline
			$\scriptstyle a_1$ & $\scriptscriptstyle 0.4$ & $\scriptscriptstyle  0.7$ & $\scriptscriptstyle 0.2$ & $\scriptscriptstyle 0.3$ \\
			$\scriptstyle a_2$ & $\scriptscriptstyle 0.5$ & $\scriptscriptstyle 0.9$ & $\scriptscriptstyle 0.1$ & $\scriptscriptstyle 0.4$\\
			$\scriptstyle a_3$ & $\scriptscriptstyle 0.6$ & $\scriptscriptstyle 0.6 $ & $\scriptscriptstyle 0.5$ & $\scriptscriptstyle 0.4$\\
			$\scriptstyle a_4$ & $\scriptscriptstyle 0.8$ & $\scriptscriptstyle 0.7$ & $\scriptscriptstyle 0.8$ & $\scriptscriptstyle 0.6$\\
			$\scriptstyle a_5$ & $\scriptscriptstyle 0.6$ & $\scriptscriptstyle 0.4$ & $\scriptscriptstyle 0.7$ & $\scriptscriptstyle0.7$\\
			\hline
		\end{tabular}
	\end{minipage}
	\hspace{.4cm}
	\begin{minipage}{.30\textwidth}
		\caption{Decision matrix  of the expert $e_2$}
		\label{Table_e2}
		\centering
		\setlength{\tabcolsep}{2pt}
		\begin{tabular}{c|cccc}
			\hline
			$\scriptstyle R^2$ & $\scriptstyle C_1$ & $\scriptstyle C_2$ & $\scriptstyle C_3$ & $\scriptstyle C_4$ \\ \hline
			$\scriptstyle a_1$ & $\scriptscriptstyle 0.5$ & $\scriptscriptstyle  0.7$ & $\scriptscriptstyle 0.5$ & $\scriptscriptstyle 0.5$ \\
			$\scriptstyle a_2$ & $\scriptscriptstyle 0.5$ & $\scriptscriptstyle 0.5$ & $\scriptscriptstyle 0.1$ & $\scriptscriptstyle 0.4$\\
			$\scriptstyle a_3$ & $\scriptscriptstyle 0.7$ & $\scriptscriptstyle 0.6 $ & $\scriptscriptstyle 0.3 $ & $\scriptscriptstyle 0.6$\\
			$\scriptstyle a_4$ & $\scriptscriptstyle 0.7$ & $\scriptscriptstyle 0.2$ & $\scriptscriptstyle 0.8$ & $\scriptscriptstyle 0.8$\\
			$\scriptstyle a_5$ & $\scriptscriptstyle 0.9$ & $\scriptscriptstyle 0.6$ & $\scriptscriptstyle 0.8$ & $\scriptscriptstyle 0.3$\\
			\hline
		 \end{tabular}
		\end{minipage}
		\hspace{.4cm}
		\begin{minipage}{.30\textwidth}
		\caption{Decision matrix  of the expert $e_3$}
		\label{Table_e3}
		\centering
		\setlength{\tabcolsep}{2pt}
		\begin{tabular}{c|cccc}
			\hline
			$\scriptstyle R^3$ & $\scriptstyle C_1$ & $\scriptstyle C_2$ & $\scriptstyle C_3$ & $\scriptstyle C_4$ \\ \hline
			$\scriptstyle a_1$ & $\scriptscriptstyle 0.4$ & $\scriptscriptstyle  0.8$ & $\scriptscriptstyle 0.2$ & $\scriptscriptstyle 0.9$ \\
			$\scriptstyle a_2$ & $\scriptscriptstyle 0.3$ & $\scriptscriptstyle 0.7$ & $\scriptscriptstyle 0.6$ & $\scriptscriptstyle 0.7$\\
			$\scriptstyle a_3$ & $\scriptscriptstyle 0.7$ & $\scriptscriptstyle 0.6$ & $\scriptscriptstyle 0.5 $ & $\scriptscriptstyle 0.4$\\
			$\scriptstyle a_4$ & $\scriptscriptstyle 0.4$ & $\scriptscriptstyle 0.4$ & $\scriptscriptstyle 0.1$ & $\scriptscriptstyle 0.8$\\
			$\scriptstyle a_5$ & $\scriptscriptstyle 0.1$ & $\scriptscriptstyle 0.6$ & $\scriptscriptstyle 0.7$ & $\scriptscriptstyle 0.6$\\
			\hline
		\end{tabular}	
		\end{minipage}
\end{table}

\begin{table}[h!]
\hspace{3.0 cm}
	\begin{minipage}{.23\textwidth}
		\caption{Decision matrix  of the expert $e_4$}
		\label{Table_e4}
		\centering
		\setlength{\tabcolsep}{2pt}
		\begin{tabular}{c|cccc}
			\hline
			$\scriptstyle R^4$ & $\scriptstyle C_1$ & $\scriptstyle C_2$ & $\scriptstyle C_3$ & $\scriptstyle C_4$ \\ \hline
			$\scriptstyle a_1$ & $\scriptscriptstyle 0.3$ & $\scriptscriptstyle  0.9$ & $\scriptscriptstyle 0.4$ & $\scriptscriptstyle 0.3$ \\
			$\scriptstyle a_2$ & $\scriptscriptstyle 0.3$ & $\scriptscriptstyle 0.2$ & $\scriptscriptstyle 0.8$ & $\scriptscriptstyle 0.3$\\
			$\scriptstyle a_3$ & $\scriptscriptstyle 0.7$ & $\scriptscriptstyle 0.9$ & $\scriptscriptstyle 0.3 $ & $\scriptscriptstyle 0.6$\\
			$\scriptstyle a_4$ & $\scriptscriptstyle 0.8$ & $\scriptscriptstyle 0.4$ & $\scriptscriptstyle 0.8$ & $\scriptscriptstyle 0.9$\\
			$\scriptstyle a_5$ & $\scriptscriptstyle 0.3$ & $\scriptscriptstyle 0.8$ & $\scriptscriptstyle 0.9$ & $\scriptscriptstyle 0.9$\\
			\hline
		\end{tabular}
	\end{minipage}
	\hspace{1.5cm}
	\begin{minipage}{.23\textwidth}
		\caption{Decision matrix  of the expert $e_5$}
		\label{Table_e5}
		\centering
		\setlength{\tabcolsep}{2pt}
		\begin{tabular}{c|cccc}
			\hline
			$\scriptstyle R^5$ & $\scriptstyle C_1$ & $\scriptstyle C_2$ & $\scriptstyle C_3$ & $\scriptstyle C_4$ \\ \hline
			$\scriptstyle a_1$ & $\scriptscriptstyle 0.5$ & $\scriptscriptstyle  0.1$ & $\scriptscriptstyle 0.5$ & $\scriptscriptstyle 0.6$ \\
			$\scriptstyle a_2$ & $\scriptscriptstyle 0.3$ & $\scriptscriptstyle 0.6$ & $\scriptscriptstyle 0.5$ & $\scriptscriptstyle 0.7$\\
			$\scriptstyle a_3$ & $\scriptscriptstyle 0.6$ & $\scriptscriptstyle 0.6$ & $\scriptscriptstyle 0.7 $ & $\scriptscriptstyle 0.6$\\
			$\scriptstyle a_4$ & $\scriptscriptstyle 0.3$ & $\scriptscriptstyle 0.7$ & $\scriptscriptstyle 0.1$ & $\scriptscriptstyle 0.6$\\
			$\scriptstyle a_5$ & $\scriptscriptstyle 0.7$ & $\scriptscriptstyle 0.7$ & $\scriptscriptstyle 0.8$ & $\scriptscriptstyle 0.6$\\
			\hline
		\end{tabular}
	\end{minipage}
\end{table}

		

If we take together the values on the same positions of the $R^k$'s , we generate  the
Table \ref{Table_Hesitant}, used as starting point in \cite{XuZ}.  Observe that some equal evaluations given by different experts collapsed in this hesitant fuzzy matrix.  The main difference between our approach and that of Xu and Zhang \cite{XuZ} is the use of $n$-dimensional fuzzy sets rather than hesitant fuzzy sets (for more details about hesitant fuzzy sets one may see \cite{Rosa,Torra} and  \cite{Bed-THFS,Bed-THFN} for typical hesitant fuzzy sets) in the decision matrix and the use of admissible orders which reduces the number of cases where alternatives with different initial evaluations lead to the same final ranking.
Now, following the proposed method, we construct  the $L_n([0,1])$-valued collective matrix  $R^*$ in Table \ref{Table_nDim} (Step 2).

\begin{table}[h!]
		\caption{Hesitant fuzzy collective decision matrix in \cite{XuZ}}
		\label{Table_Hesitant}
		\centering
		\setlength{\tabcolsep}{2pt}
		\begin{tabular}{c|cccc}
			\hline
			 $\scriptstyle R$ & $\scriptstyle C_1$ & $\scriptstyle C_2$ & $\scriptstyle C_3$ & $\scriptstyle C_4$ \\ \hline
			$\scriptstyle a_1$ & $\scriptscriptstyle \{0.5, 0.4, 0.3\}$ & $\scriptscriptstyle \{0.9, 0.8, 0.7, 0.1\}$ & $\scriptscriptstyle \{0.5, 0.4, 0.2\}$ & $\scriptscriptstyle \{0.9, 0.6, 0.5, 0.3\}$ \\
			$\scriptstyle a_2$ & $\scriptscriptstyle \{0.5, 0.3\}$ & $\scriptscriptstyle \{0.9,0.7,0.6,0.5,0.2\}$ & $\scriptscriptstyle \{0.8,0.6,0.5,0.1\}$ & $\scriptscriptstyle \{0.7,0.4,0.3\}$\\
			$\scriptstyle a_3$ & $\scriptscriptstyle \{0.7,0.6\}$ & $\scriptscriptstyle \{0.9,0.6\}$ & $\scriptscriptstyle \{0.7,0.5,0.3\}$ & $\scriptscriptstyle \{0.6,0.4\}$\\
			$\scriptstyle a_4$ & $\scriptscriptstyle \{0.8,0.7,0.4,0.3\}$ & $\scriptscriptstyle \{0.7,0.4,0.2\}$ & $\scriptscriptstyle \{0.8,0.1\}$ & $\scriptscriptstyle \{0.9,0.8,0.6\}$\\
			$\scriptstyle a_5$ & $\scriptscriptstyle \{0.9,0.7,0.6,0.3,0.1\}$ & $\scriptscriptstyle \{0.8,0.7,0.6,0.4\}$ & $\scriptscriptstyle \{0.9,0.8,0.7\}$ & $\scriptscriptstyle \{0.9,0.7,0.6,0.3\}$\\
			\hline
		\end{tabular}
\end{table}
\begin{table}[h!]
	\caption{$n$-dimensional interval collective decision matrix}
	\label{Table_nDim}
	\centering
	\setlength{\tabcolsep}{2pt}
	\begin{tabular}{c|cccc}
		\hline
		$\scriptstyle R^*$ & $\scriptstyle C_1$ & $\scriptstyle C_2$ & $\scriptstyle C_3$ & $\scriptstyle C_4$ \\ \hline
		$\scriptstyle a_1$ & $\scriptscriptstyle (0.3, 0.4, 0.4, 0.5,0.5)$ & $\scriptscriptstyle (0.1,0.7, 0.7, 0.8 , 0.9)$ & $\scriptscriptstyle (0.2, 0.2, 0.4, 0.5,0.5)$ & $\scriptscriptstyle (0.3, 0.3, 0.5, 0.6, 0.9)$ \\
		$\scriptstyle a_2$ & $\scriptscriptstyle (0.3,0.3,0.3, 0.5, 0.5)$ & $\scriptscriptstyle (0.2, 0.5, 0.6, 0.7, 0.9)$ & $\scriptscriptstyle (0.1, 0.1, 0.5, 0.6, 0.8)$ & $\scriptscriptstyle (0.3, 0.4, 0.4,0.7,0.7)$\\
		$\scriptstyle a_3$ & $\scriptscriptstyle (0.6,0.6,0.7, 0.7, 0.7)$ & $\scriptscriptstyle (0.6, 0.6,0.6,0.6,0.9)$ & $\scriptscriptstyle (0.3,0.3,0.5, 0.5, 0.7)$ & $\scriptscriptstyle (0.4, 0.4,0.6,0.6,0.6)$\\
		$\scriptstyle a_4$ & $\scriptscriptstyle (0.3, 0.4, 0.7, 0.8,0.8)$ & $\scriptscriptstyle (0.2, 0.4, 0.4,0.7, 0.7)$ & $\scriptscriptstyle (0.1, 0.1,0.8,0.8,0.8)$ & $\scriptscriptstyle (0.6,0.6, 0.8,0.8, 0.9)$\\
		$\scriptstyle a_5$ & $\scriptscriptstyle (0.1, 0.3, 0.6, 0.7, 0.9)$ & $\scriptscriptstyle (0.4, 0.6, 0.6, 0.7, 0.8)$ & $\scriptscriptstyle (0.7,0.7, 0.8,0.8, 0.9)$ & $\scriptscriptstyle (0.3, 0.6,0.6, 0.7, 0.9)$\\
		\hline
	\end{tabular}
\end{table}

The selected weighting vector (Step 3) is $\omega = (0.2341, 0.2474, 0.3181, 0.2004)$ (as in \cite{XuZ}) and the admissible order on $\mathcal{L}_n([0,1])$ chosen (Step 4) is $\preceq_\tau$ (as in Example \ref{ex-adm-orders}) for $\tau(1)=3, \tau(2)=2, \tau(3)=4, \tau(4)=1$ and $\tau(5)=5$ (notice that for this illustrative example the choice of $\tau$ is ad-hoc). The $m$-ary $n$-dimensional aggregation function with respect to $\preceq_\tau$ chosen in the step 4  is $\mathcal{M}_\omega^{\mathcal{L}_n([0,1])}$ defined in Equation (\ref{eq-Mw-preceq}) where the semi-vector space $\mathcal{L}_n([0,1])$ is the given in Theorem \ref{teo-SVS}.
In the step 5,  $\mathcal{M}_\omega^{\mathcal{L}_n([0,1])}$ is applied to each row  of the $R^*$ resulting in the following  $L_n([0,1])$-scores for each alternative:

{\setlength\arraycolsep{1pt}
\begin{eqnarray*}
	s_1 & = & \mathcal{M}_\omega^{\mathcal{L}_n([0,1])}(R^*_{11}, \dots, R^*_{14}) \\
	& = &  0.2341 \odot (0.3, 0.4, 0.4, 0.5,0.5) \circplus 0.2474 \odot (0.1,0.7, 0.7, 0.8 , 0.9) \circplus \\
	& & 0.3181 \odot (0.2, 0.2, 0.4, 0.5,0.5)\circplus 0.2004 \odot (0.3, 0.3, 0.5, 0.6, 0.9) \\
	& = & (0.07023,0.09364,0.09364,0.11705,0.11705)\circplus (0.02474,0.17318,0.17318,0.19792,0.22266) \circplus\\
	&  & (0.06362,0.06362,0.12724,0.15905,0.15905)\circplus (0.06012,0.06012,0.1002,0.12024,0.18036) \\
	& = & (0.21871,0.39056,0.49426,0.59426,0.67912)
	\\
	\end{eqnarray*}
	
	\begin{eqnarray*}
	s_2 & = & \mathcal{M}_\omega^{\mathcal{L}_n([0,1])}(R^*_{21}, \dots, R^*_{24}) \\
& = &  0.2341 \odot (0.3,0.3,0.3, 0.5, 0.5) \circplus 0.2474 \odot (0.2, 0.5, 0.6, 0.7, 0.9) \circplus \\
	& & 0.3181 \odot (0.1, 0.1, 0.5, 0.6, 0.8)\circplus 0.2004 \odot (0.3, 0.4, 0.4,0.7,0.7)\\
	& = & (0.07023,0.07023,0.07023,0.11705,0.11705)\circplus (0.04948,0.1237, 0.19792, 0.17318,0.22266) \circplus\\
	&  & (0.03181,0.03181,0.15905,0.19086,0.25448)\circplus (0.06012,0.08016,0.08016,0.14028,0.14028) \\
	& = & (0.21164,0.3059,0.50736,0.62137,0.73447) \\ \\
	s_3 & = & \mathcal{M}_\omega^{\mathcal{L}_n([0,1])}(R^*_{31}, \dots, R^*_{34}) \\
	& = &  0.2341 \odot (0.6,0.6,0.7, 0.7, 0.7) \circplus 0.2474 \odot (0.6, 0.6,0.6,0.6,0.9) \circplus \\
	& & 0.3181 \odot (0.3,0.3,0.5, 0.5, 0.7)\circplus 0.2004 \odot (0.4, 0.4,0.6,0.6,0.6) \\
	& = & (0.14046,0.14046,0.16387,0.16387,0.16387)\circplus (0.14844,0.14844,0.14844,0.14844,0.22266) \circplus  \\
	&  &  (0.09543,0.09543,0.15905,0.15905,0.22267)\circplus (0.08016,0.08016,0.12024,0.12024,0.12024)  \\
	& = &  (0.46449,0.46449,0.5916,0.5916,0.72944) \\ \\
	s_4 & = & \mathcal{M}_\omega^{\mathcal{L}_n([0,1])}(R^*_{41}, \dots, R^*_{44}) \\
	& = &  0.2341 \odot (0.3, 0.4, 0.7, 0.8,0.8) \circplus 0.2474 \odot (0.2, 0.4, 0.4,0.7, 0.7) \circplus \\
	& & 0.3181 \odot (0.1, 0.1,0.8,0.8,0.8)\circplus 0.2004 \odot (0.6,0.6, 0.8,0.8, 0.9) \\
	& = &  (0.07023,0.09364,0.16387,0.18728,0.18728)\circplus (0.03181,0.03181,0.09896,0.17318,0.17318) \circplus \\
	&  & (0.06362,0.06362,0.25448,0.25448,0.25448)\circplus (0.12024,0.12024,0.16032,0.16032,0.18036)  \\
	& = & (0.2859,0.30931,0.67763,0.77526,0.7953)
	\\\\
	s_5 & = & \mathcal{M}_\omega^{\mathcal{L}_n([0,1])}(R^*_{51}, \dots, R^*_{54}) \\
	& = &  0.2341 \odot (0.1, 0.3, 0.6, 0.7, 0.9) \circplus 0.2474 \odot (0.4, 0.6, 0.6, 0.7, 0.8) \circplus \\
	& & 0.3181 \odot (0.7,0.7, 0.8,0.8, 0.9))\circplus 0.2004 \odot (0.3, 0.6,0.6, 0.7, 0.9) \\
	& = & (0.02341,0.07023,0.14046,0.16387,0.21069)\circplus (0.09896,0.14844,0.14844,0.17318,0.19792) \circplus\\
	&  & (0.22267,0.22267,0.25448,0.25448,0.28629)\circplus (0.06012,0.12024,0.12024,0.14028,0.18036)  \\
	& = & (0.40516,0.56158,0.66362,0.73181,0.87526)
\end{eqnarray*}}

Thereby, we obtain the following ranking of the alternatives: $a_1 < a_2 < a_3 < a_5 < a_4$ (Step 6), since $s_1 \preceq_\tau s_2 \preceq_\tau s_3 \preceq_\tau s_5 \preceq_\tau s_4$; and thus the most desirable alternative is $a_4$, which differs from the best alternative obtained in \cite{XuZ}.
The problem with applying the method of Xu and Zhang, is that when two experts provide the same degree for the satisfaction of an alternative with respect to a criterion, they collapse in a single value in the hesitant decision matrix as, for example, with the evaluation of the alternative
$a_1$ with respect to the criterion $C_1$ done by the experts $e_1$ and $e_3$ (see Tables \ref{Table_e1} and \ref{Table_e3}).
Moreover, in the step 2 of the Xu and Zhang method, the least value is repeated in order to obtain vectors on $[0,1]$ of the same dimension. Thus, this process does not faithfully reflect the opinion of the experts, and therefore the ranking obtained will be biased. Indeed, if the expert $e_2$ changes the evaluation for the alternative $a_4$ with respect to $C_3$  to $0.1$, this change will not have any effect over the hesitant decision matrix in Table \ref{Table_Hesitant}, and therefore the final ranking will remain the same, whereas in the proposed method, the ranking will change.


\section{Conclusion}

In this paper, $m$-ary aggregation functions with respect to admissible orders on $n$-dimensional intervals, i.e. on $L_n([0,1])$, were studied,. Different classes of such functions have been considered, such as conjunctive, disjunctive, average, idempotent and symmetric  $m$-ary $n$-dimensional aggregation functions with respect to an arbitrary admissible order. On the other hand, the notion of ordered semi-vector spaces on the semifield of non-negative real numbers was extended for arbitrary weak semifields. In particular, an ordered semi-vector space for $L_n([0,1])$ where the order is admissible is given and OWA-like and weigthed average operators for $L_n([0,1])$ are defined. Then, we have investigated on admissible orders such that these operators are aggregation functions.
Finally, we have developed an application in a multicriteria group decision making method based on $n$-dimensional aggregation functions with respect to an admissible order. Moreover, an  illustrative example is considered where  the $m$-ary $n$-dimensional weigthed average aggregation function for a given admissible order is applied.  Some minimal desirable properties that all multiple criteria group decision making method must satisfy and that are preserved by the  proposed method are discussed.

As further work, we will investigate on new classes and on the way of constructing admissible order on $L_n([0,1])$. We also intend to study  $m$-ary $n$ dimensional aggregation  functions with respect to such admissible orders.  On the other hand, continuity is an important property in many applications, since it guarantees that small changes in the inputs of an aggregation function do not imply substantives changes of the results. So, in a future work a deeper study of continuous $n$-dimensional aggregation functions will be made, not only in the light of \cite{BedregalMR18,MezzomoBM18}, but also taking into account the admissible order which is being considered.

\section*{Acknowledgments} This work is partially supported by the  Brazilian National Council for Scientific and Technological Development CNPq under the Process 311429/2020-3 and by the Spanish Government through Research Project PID2019-108392GB-I00 \linebreak (AEI/10.13039/501100011033)

\end{document}